\documentclass[11pt,reqno]{amsart}

\usepackage{amsmath,amsfonts,amsthm,amssymb}
\usepackage{graphicx}
\usepackage{verbatim}
\usepackage{color}
\usepackage{amscd}
\usepackage{verbatim}
\usepackage{bigints}

\begin{document}
\newcommand{\M}{{\mathcal M}}
\newcommand{\loc}{{\mathrm{loc}}}
\newcommand{\core}{C_0^{\infty}(\Omega)}
\newcommand{\sob}{W^{1,p}(\Omega)}
\newcommand{\sobloc}{W^{1,p}_{\mathrm{loc}}(\Omega)}
\newcommand{\merhav}{{\mathcal D}^{1,p}}
\newcommand{\be}{\begin{equation}}
\newcommand{\ee}{\end{equation}}
\newcommand{\mysection}[1]{\section{#1}\setcounter{equation}{0}}
\newcommand{\laplace}{\Delta}
\newcommand{\pl}{\laplace_p}
\newcommand{\grad}{\nabla}
\newcommand{\pd}{\partial}
\newcommand{\bo}{\pd}
\newcommand{\csub}{\subset \subset}
\newcommand{\sm}{\setminus}
\newcommand{\ssm}{:}
\newcommand{\diver}{\mathrm{div}\,}
\newcommand{\bea}{\begin{eqnarray}}
\newcommand{\eea}{\end{eqnarray}}
\newcommand{\bean}{\begin{eqnarray*}}
\newcommand{\eean}{\end{eqnarray*}}
\newcommand{\thkl}{\rule[-.5mm]{.3mm}{3mm}}
\newcommand{\cw}{\stackrel{\rightharpoonup}{\rightharpoonup}}
\newcommand{\id}{\operatorname{id}}
\newcommand{\supp}{\operatorname{supp}}
\newcommand{\wlim}{\mbox{ w-lim }}
\newcommand{\mymu}{{x_N^{-p_*}}}
\newcommand{\R}{{\mathbb R}}
\newcommand{\N}{{\mathbb N}}
\newcommand{\Z}{{\mathbb Z}}
\newcommand{\Q}{{\mathbb Q}}
\newcommand{\abs}[1]{\lvert#1\rvert}
\newtheorem{theorem}{Theorem}[section]
\newtheorem{corollary}[theorem]{Corollary}
\newtheorem{lemma}[theorem]{Lemma}
\newtheorem{notation}[theorem]{Notation}
\newtheorem{definition}[theorem]{Definition}
\newtheorem{remark}[theorem]{Remark}
\newtheorem{proposition}[theorem]{Proposition}
\newtheorem{assertion}[theorem]{Assertion}
\newtheorem{problem}[theorem]{Problem}
\newtheorem{conjecture}[theorem]{Conjecture}
\newtheorem{question}[theorem]{Question}
\newtheorem{example}[theorem]{Example}
\newtheorem{Thm}[theorem]{Theorem}
\newtheorem{Lem}[theorem]{Lemma}
\newtheorem{Pro}[theorem]{Proposition}
\newtheorem{Def}[theorem]{Definition}
\newtheorem{Exa}[theorem]{Example}
\newtheorem{Exs}[theorem]{Examples}
\newtheorem{Rems}[theorem]{Remarks}
\newtheorem{Rem}[theorem]{Remark}

\newtheorem{Cor}[theorem]{Corollary}
\newtheorem{Conj}[theorem]{Conjecture}
\newtheorem{Prob}[theorem]{Problem}
\newtheorem{Ques}[theorem]{Question}
\newtheorem*{corollary*}{Corollary}
\newtheorem*{theorem*}{Theorem}
\newcommand{\pf}{\noindent \mbox{{\bf Proof}: }}

\renewcommand{\theequation}{\thesection.\arabic{equation}}
\catcode`@=11 \@addtoreset{equation}{section} \catcode`@=12
\newcommand{\Real}{\mathbb{R}}
\newcommand{\real}{\mathbb{R}}
\newcommand{\Nat}{\mathbb{N}}
\newcommand{\ZZ}{\mathbb{Z}}
\newcommand{\CC}{\mathbb{C}}
\newcommand{\Pess}{\opname{Pess}}
\newcommand{\Proof}{\mbox{\noindent {\bf Proof} \hspace{2mm}}}
\newcommand{\mbinom}[2]{\left (\!\!{\renewcommand{\arraystretch}{0.5}
\mbox{$\begin{array}[c]{c}  #1\\ #2  \end{array}$}}\!\! \right )}
\newcommand{\brang}[1]{\langle #1 \rangle}
\newcommand{\vstrut}[1]{\rule{0mm}{#1mm}}
\newcommand{\rec}[1]{\frac{1}{#1}}
\newcommand{\set}[1]{\{#1\}}
\newcommand{\dist}[2]{$\mbox{\rm dist}\,(#1,#2)$}
\newcommand{\opname}[1]{\mbox{\rm #1}\,}
\newcommand{\mb}[1]{\;\mbox{ #1 }\;}
\newcommand{\undersym}[2]
 {{\renewcommand{\arraystretch}{0.5}  \mbox{$\begin{array}[t]{c}
 #1\\ #2  \end{array}$}}}
\newlength{\wex}  \newlength{\hex}
\newcommand{\understack}[3]{%
 \settowidth{\wex}{\mbox{$#3$}} \settoheight{\hex}{\mbox{$#1$}}
 \hspace{\wex}  \raisebox{-1.2\hex}{\makebox[-\wex][c]{$#2$}}
 \makebox[\wex][c]{$#1$}   }%
\newcommand{\smit}[1]{\mbox{\small \it #1}}
\newcommand{\lgit}[1]{\mbox{\large \it #1}}
\newcommand{\scts}[1]{\scriptstyle #1}
\newcommand{\scss}[1]{\scriptscriptstyle #1}
\newcommand{\txts}[1]{\textstyle #1}
\newcommand{\dsps}[1]{\displaystyle #1}
\newcommand{\dx}{\,\mathrm{d}x}
\newcommand{\dy}{\,\mathrm{d}y}
\newcommand{\dz}{\,\mathrm{d}z}
\newcommand{\dt}{\,\mathrm{d}t}
\newcommand{\dr}{\,\mathrm{d}r}
\newcommand{\du}{\,\mathrm{d}u}
\newcommand{\dv}{\,\mathrm{d}v}
\newcommand{\dV}{\,\mathrm{d}V}
\newcommand{\ds}{\,\mathrm{d}s}
\newcommand{\dS}{\,\mathrm{d}S}
\newcommand{\dk}{\,\mathrm{d}k}

\newcommand{\dphi}{\,\mathrm{d}\phi}
\newcommand{\dtau}{\,\mathrm{d}\tau}
\newcommand{\dxi}{\,\mathrm{d}\xi}
\newcommand{\deta}{\,\mathrm{d}\eta}
\newcommand{\dsigma}{\,\mathrm{d}\sigma}
\newcommand{\dtheta}{\,\mathrm{d}\theta}
\newcommand{\dnu}{\,\mathrm{d}\nu}
\newcommand{\Bint}{\bigintsss}

\def\ga{\alpha}     \def\gb{\beta}       \def\gg{\gamma}
\def\gc{\chi}       \def\gd{\delta}      \def\ge{\epsilon}
\def\gth{\theta}                         \def\vge{\varepsilon}
\def\gf{\phi}       \def\vgf{\varphi}    \def\gh{\eta}
\def\gi{\iota}      \def\gk{\kappa}      \def\gl{\lambda}
\def\gm{\mu}        \def\gn{\nu}         \def\gp{\pi}
\def\vgp{\varpi}    \def\gr{\rho}        \def\vgr{\varrho}
\def\gs{\sigma}     \def\vgs{\varsigma}  \def\gt{\tau}
\def\gu{\upsilon}   \def\gv{\vartheta}   \def\gw{\omega}
\def\gx{\xi}        \def\gy{\psi}        \def\gz{\zeta}
\def\Gg{\Gamma}     \def\Gd{\Delta}      \def\Gf{\Phi}
\def\Gth{\Theta}
\def\Gl{\Lambda}    \def\Gs{\Sigma}      \def\Gp{\Pi}
\def\Gw{\Omega}     \def\Gx{\Xi}         \def\Gy{\Psi}

\newcommand{\cqfd}{\begin{flushright} 
                    $\Box$
                   \end{flushright}}

\renewcommand{\div}{\mathrm{div}}
\newcommand{\red}[1]{{\color{red} #1}}


\title[Hardy spaces and heat kernel regularity]{Hardy spaces and heat kernel regularity}

\author{Baptiste Devyver}
\address{Baptiste Devyver, Department of Mathematics,  Technion - Israel Institute of Technology, Haifa 32000, Israel}
\email{baptiste.devyver@univ-nantes.fr}

 \maketitle
\newcommand{\dnorm}[1]{\thkl #1 \thkl\,}

\begin{abstract}

In this paper, we show the equivalence between the boundedness of the Riesz transform $d\Delta^{-1/2}$ on $L^p$, $p\in (2,p_0)$, and the equality $H^p=L^p$, $p\in(2,p_0)$, in the class of manifold whose measure is doubling and for which the scaled Poincar\'{e} inequalities hold. Here, $H^p$ is a Hardy space of exact $1-$forms, naturally associated with the Riesz transform.

\end{abstract}


\section{Introduction}

In this article, we will be concerned with questions related to the boundedness of the Riesz transform on manifolds. Since the seminal work of Coulhon and Duong \cite{CD}, who gave sufficient conditions on the heat kernel so that the Riesz transform is bounded on $L^p$ for $1<p\leq 2$, several authors have investigated both necessary and sufficient conditions for the boundedness of the Riesz transform on manifolds. For $p>2$, one of the main achievements is the following result due to Auscher, Coulhon, Duong and Hofmann \cite{ACDH}: if the manifold satisfies the scaled Poincar\'{e} inequalities and the Riemannian measure is doubling, then the boundedness of the Riesz transform on $L^q$ for $q\in (2,p)$ is equivalent to the following bounds on the gradient of the heat kernel: for every $q\in (2,p)$,

\begin{equation}\label{grad}
||\nabla e^{-t^2\Delta}||_{q,q}\leq \frac{C_q}{t},\qquad\forall t>0.
\end{equation}
However, the question to find meaningful geometric conditions so that the gradient estimates for the heat kernel \eqref{grad} hold is a difficult problem, and not much is known about it, although some progress has been recently made in the understanding of inequalities that are stronger than \eqref{grad} -- the so-called Gaussian estimates for the heat kernel of the Hodge Laplacian acting on $1-$forms (see \cite{CZ}, \cite{D}). See also \cite{AC} where the gradient estimates \eqref{grad} are proved to be equivalent under some conditions to reverse H\"{o}lder inequalities for the gradient of harmonic functions).\\

Very recently, another way to tackle the problem of boundedness of the Riesz transform has been investigated: it consists in introducing Hardy spaces of forms $H^p$ for $1\leq p\leq \infty$, which are adapted to the problem in the sense that the Riesz transform is \textit{always} bounded from  $H^p$ to $H^p$ for every $1\leq p\leq \infty$. Several authors have independantly performed such a construction: see \cite{AMR} for the case of Hardy spaces associated the Hodge Laplacian on a manifold, and \cite{HLMMY} for the case of Hardy spaces associated to second-order elliptic operators on $\R^n$. Let us emphasize that the construction of the Hardy spaces requires that the Riemannian measure be doubling. A natural question (asked by the authors in \cite{AMR}) is the following:

\begin{Prob}

Let $1<p<\infty$. Under which conditions $H^p=L^p$?

\end{Prob}
It is shown in \cite{AMR}, Corollary 1.2 that if the Riemannian measure is doubling, then for every $2\leq p<\infty$, 

$$L^p\subset H^p.$$
Consequently, for every $2\leq p<\infty$, the Riesz transform is bounded $L^p\rightarrow H^p$. But it could happen \textit{a priori} that this Hardy space $H^p$ is ``too big", that is that one could find an intermediate space $L^p\varsubsetneqq \tilde{H}^p \varsubsetneqq H^p$, on which the Riesz transform is bounded. One of the main results of this article provides a hint that this should not happen, and that indeed $H^p$ is the correct space to be considered. The result can be very roughly stated as follows:

\begin{Thm}\label{main}{\em (Main result)} Assume that $M$ is a connected, complete, non-compact Riemannian manifold such that the Riemannian measure is doubling and the scaled Poincar\'{e} inequalities hold. Then for every $2<p<\infty$, the following are equivalent:

\begin{enumerate}

 \item For every $q\in (2,p)$, $H^q=L^q$ with equivalent norms.
 
 \item For every $q\in (2,p)$, the Riesz transform $d\Delta^{-1/2}$ is bounded on $L^q$.

\end{enumerate}
If one of these two equivalent conditions is satisfied, then for every $q\in (2,p)$,

$$||u||_q\simeq ||u||_{H^q}\simeq ||d\Delta^{-1/2}u||_q,\qquad \forall u\in C_0^\infty(M).$$

\end{Thm}

\begin{Rem}
{\em

The precise  formulation of Theorem \ref{main} is given in Corollary \ref{main2}. Here, we just mention that the space $H^q$ is a Hardy space of exact $1-$forms, and the equality $H^q=L^q$ should be understood as the fact that $H^q$ is equal to the closure in $L^q$ of the space of $L^2\cap L^q$ exact $1-$forms.

}
\end{Rem}
The last part of Theorem \ref{main} is reminiscent of the classical result that on $\R^n$, if $p\geq1$,

$$||u||_{H^p}\simeq ||d\Delta^{-1/2}u||_p$$
(see \cite{St}, Proposition 3, p.123), and supports the claim that indeed $H^p$ is the good space to be considered. Let us further mention that in general, one has only

$$||u||_{H^q}\simeq ||d\Delta^{-1/2}u||_{H^q},\qquad \forall u\in C_0^\infty(M)$$
(see \cite{AMR}, Theorem 5.11). We ask the following open question:

\begin{Prob}

Does the result of Theorem \ref{main} remain true without the assumption that $M$ satisfies the scaled Poincar\'{e} inequalities?

\end{Prob}
The proof of Theorem \ref{main} consists in showing that the gradient estimates for the heat kernel \eqref{grad} are sufficient to have the equality $H^p=L^p$, and relies heavily on techniques developped in \cite{ACDH}. Let us mention that it is claimed in \cite{AMR}, Theorem 8.5 that $H^p=L^p$ if the heat kernel of the Hodge Laplacian acting on differential $1-$forms satisfies a Gaussian estimate -- in fact, the proof relies on an unpublished manuscript of Auscher, Duong and McIntosh. Under the additional hypothesis that the heat kernel satisfies Gaussian upper-bounds, we can recover this result as a corollary of Theorem \ref{main}.\\

The plan of this article is as follows: in Section 2, we introduce the notations and definitions that will be needed (in particular, we recall the definition of the Hardy spaces). In Section 3, we present our results in greater details. In Section 4, we give the proof of a result concerning the area functional. In Section 5, we prove Theorem \ref{main}. 


\section{Preliminaries}

In this section, we recall some definitions and notations about the tent spaces and Hardy spaces for the Laplacian and the Dirac operator. References for this are \cite{CMS} and \cite{AMR}. 

\subsection{Notations and definitions}

For two positive functions $f$ and $g$, we write 

$$f\simeq g$$
if there is a constant $C$ such that $C^{-1}g\leq f\leq Cg$.\\ 

In all the article, $(M,g)$ will be a complete Riemannian manifold. We will denote by $\mathrm{d}x$ the Riemannian measure on $M$. The measure of a measurable set $E$ will be simply denoted by $|E|$. For $x\in M$ and $r>0$, we denote by $V(x,r)$ the measure of the geodesic ball $B(x,r)$ of center $x$ and radius $r$. We will say that the measure on $M$ is \textit{doubling} if there exists a constant $C>0$ such that

\begin{equation}\label{D}\tag{D}
V(x,2r)\leq CV(x,r),
\end{equation}
for every $x\in M$ and every $r>0$. If $\mu$ is doubling, then there exists $\kappa>0$ and $C$ such that

$$V(x,\theta r)\leq C\theta^\kappa V(x,r),$$
for every $x\in M$, $r>0$ and $\theta>1$. We will denote by $\beta$ the smallest integer strictly greater than $\frac{\kappa}{2}$. The fact that the measure is doubling has also the following consequence, which is classical:

\begin{Lem}\label{equiv_vol}

Assume \eqref{D}. Then for every $\varepsilon>0$, there is a constant $C_\varepsilon$ such that

$$C_\varepsilon^{-1}V(x,t)\leq V(y,t)\leq C_\varepsilon V(x,t),$$
for every $x, y\in M$ and $t>0$ such that $\frac{d(x,y)}{t}\leq \varepsilon$.

\end{Lem}
We will say that $M$ satisfies the \textit{scaled Poincar\'{e} inequalities} if there is a constant $C$ such that for every ball $B$ of radius $r$, and every function $f\in C^\infty(B)$,

\begin{equation}\label{P}\tag{P}
\int_B|f-f_B|^2\leq Cr^2\int_B|\nabla f|^2,
\end{equation}
where $f_B=\frac{1}{|B|}\int_B f$ is the average of $f$ on $B$. It is known that \eqref{D} together with \eqref{P} are equivalent to the so-called Li-Yau estimates for the heat kernel:

\begin{equation}\label{LY}\tag{LY}
\frac{C_1}{V(x,t)}e^{-c_1\frac{d^2(x,y)}{t}}\leq e^{-t\Delta}(x,y)\leq \frac{C_2}{V(x,t)}e^{-c_2\frac{d^2(x,y)}{t}},
\end{equation}
for every $x,y\in M$ and $t>0$. We will also consider the Gaussian estimates for the heat kernel associated to the Hodge Laplacian $\Delta_k=dd^\star+d^\star d$, acting on $k-$forms:

\begin{equation}\label{G_k}\tag{$G_k$}
|e^{-t\Delta_k}(x,y)|\leq \frac{C}{V(x,t)}e^{-\frac{d^2(x,y)}{t}},
\end{equation}
for every $x,y\in M$ and $t>0$. In particular, \eqref{D} together with \eqref{P} imply \eqref{G_k} for $k=0$. We now introduce the $L^p$ gradient estimates for the heat kernel:

\begin{equation}\label{grad_p}\tag{$\nabla_p$}
||\nabla e^{-t^2\Delta}||_{p,p}\leq \frac{C}{t},\qquad\forall t>0.
\end{equation}
For more details about inequality \eqref{grad_p}, see \cite{ACDH}. We simply remark that, according to \cite{ACDH}, Gaussian estimates for the heat kernel of both $\Delta$ and $\Delta_1$ (that is, \eqref{G_k} for $k=0,1$), imply that \eqref{grad_p} holds for every $2\leq p<\infty$. Although recently some progress has been made in finding conditions under which \eqref{G_k} for $k=1$ holds (see \cite{CZ}, \cite{D}), the corresponding problem for the weaker gradient estimates \eqref{grad_p} is so far completely open.\\

We recall some $L^2$ off-diagonal estimates -- the so-called ``Davies-Gaffney estimates" -- for the gradient of the heat kernel (these estimates originates, in the case of the heat kernel itself, with the works of Davies and Gaffney, see \cite{ACDH} for references).

\begin{Lem}\label{L2-off}

There exists two positive constants $C$ and $c$ such that for every subsets $E,F\subset M$ and every $t>0$,

$$||\chi_F t\nabla e^{-t^2\Delta}\chi_E||_{2,2}\leq C e^{-c\frac{d^2(E,F)}{t^2}}.$$

\end{Lem}
For a proof, see \cite{ACDH}, p.21. Finally, we state without proof a technical result, which is a slight variation on Lemma 3.2 in \cite{ACDH} and is proved using $L^p$ off-diagonal estimates for the gradient of the heat kernel. For a ball $B$ of radius $r$, let $C_1=4B$, and for $i\geq2$, $C_i=2^{i+1}B\setminus 2^iB$. For a given function $f$, define $f_i=\chi_{C_i}f$.

\begin{Lem}[\cite{ACDH}, Lemma 3.2]\label{ultra}

Assume \eqref{D}, \eqref{G_k} for $k=0$ and \eqref{grad_p}. Then for every $q\in [2,p)$,

\begin{equation}\label{1_ultra_p}
\left(\frac{1}{|B|}\int_B |t\nabla e^{-t^2\Delta} f_1|^q\right)^{1/q}\leq C\left(\frac{1}{|4B|}\int_{4B}f^2\right)^{1/2},
\end{equation}
and

\begin{equation}\label{i_ultra_p}
\left(\frac{1}{|B|}\int_B |t\nabla e^{-t^2\Delta} f_i|^q\right)^{1/q}\leq Ce^{-c4^i\left(\frac{r}{t}\right)^2}\left(\frac{1}{|2^{i+1}B|}\int_{C_i}f^2\right)^{1/2}.
\end{equation}
Furthermore, for $q=2$ the corresponding inequalities hold without assuming \eqref{grad_p}.

\end{Lem}

\subsection{Tent spaces}

For $x\in M$ and $\alpha>0$, the cone $\Gamma_\alpha(x)$ of aperture $\alpha$ and vertex $x$ is defined by

$$\Gamma_\alpha(x):=\{(t,y)\in (0,\infty)\times M\,:\,y\in B(x,\alpha t)\}.$$
For $\alpha=1$, we will denote $\Gamma(x)$ instead of $\Gamma_1(x)$. For a closed set $F\subset M$, let $\mathcal{R}(F)$ be the union of all cones $\Gamma(x)$ with vertex $x\in F$. If $O\subset M$ is an open set, the tent over $O$, denoted $T(O)$, is the complement in $M\times (0,\infty)$ of $\mathcal{R}(M\setminus O)$.\\

For a Hilbert space $\mathcal{H}$, and a family $F=(F_t)_{t>0}$ of measurable functions from $M$ to $\mathcal{H}$, define

$$\mathcal{A}_\alpha F(x)=\left(\int_{\Gamma_\alpha(x)} |F(t,y)|^2\frac{dy}{V(x,t)}\frac{dt}{t}\right)^{1/2}.$$
Here, $|F(t,y)|$ is the norm of $F(t,y):=F_t(y)$ in $\mathcal{H}$. When $\alpha=1$, we will write $\mathcal{A}$ instead of $\mathcal{A}_1$. For $1\leq p<\infty$, we define $T^{2,p}(\mathcal{H})$ the set (modulo equivalent classes) of such families $F$ such that $\mathcal{A}F\in L^p(M)$, equipped with the norm $||F ||_{T^{2,p}(\mathcal{H})}=||\mathcal{A}F||_{L^p(M)}$. It is known (see \cite{CMS}) that if the measure on $M$ is doubling \eqref{D}, then a different choice of $\alpha$ gives rise to equivalent norms, that is

$$||\mathcal{A}_\alpha F||_p\simeq ||\mathcal{A}F||_p,$$
for every $1\leq p<\infty$. From now on, we will write $T^{2,p}$ instead of $T^{2,p}(\mathcal{H})$ (in practice, the choice of $\mathcal{H}$ will be obvious). Let us remark that in the case $p=2$, if one assumes the doubling property \eqref{D}, then the norm in $T^{2,2}$ is equivalent to

$$||F||_{T^{2,2}}\simeq \left(\int_0^\infty \int_M |F(t,x)|^2\,dx\frac{dt}{t}\right)^{1/2}.$$
In fact, one has the following more precise result:

\begin{Lem}\label{tent_sup}

Assume \eqref{D}. Then for every $F$,

\begin{equation}\label{T22_equiv}
||F||_{T^{2,2}}\simeq \left(\int_0^\infty\int_M |F(t,y)|^2\,dy\frac{dt}{t}\right)^{1/2}.
\end{equation}
If, moreover, for some ball $B$ of radius $r$, $F$ is supported in $T(B)$, then

\begin{equation}\label{tent_sup2}
\int_B|\mathcal{A}F|^2\simeq \int_{T(B)}|F(t,y)|^2\,dy\frac{dt}{t},
\end{equation}
and for every $p>2$,

\begin{equation}\label{tent_sup_p}
\left(\int_B|\mathcal{A}F|^p\right)^{2/p}\leq C\int_0^r\frac{dt}{t} \left(\int_{B} |F(t,y)|^p\,dy \right)^{2/p}.
\end{equation}

\end{Lem}
 
\begin{proof}
By Lemma \ref{equiv_vol}, if $d(x,y)\leq t$, then

$$V(x,t)\simeq V(y,t).$$
Thus, using the fact that for $(t,y)$ fixed,

$$\{x\in M\,;\,(t,y)\in \Gamma(x)\}=B(y,t),$$
there holds

$$\begin{array}{rcl}||F||^2_{T^{2,2}}=||\mathcal{A}F||^2_2&=&\int_M dx\int_{\Gamma(x)} |F(t,y)|^2 \frac{dy}{V(x,t)}\frac{dt}{t}\\\\
&\simeq& \int_M dx\int_{\Gamma(x)} |F(t,y)|^2 \frac{dy}{V(y,t)}\frac{dt}{t}\\\\
&\simeq& \int_0^\infty \frac{dt}{t}\int_M |F(t,y)|^2 \frac{V(y,t)}{V(y,t)} dy\\\\
&\simeq& \int_0^\infty\int_M |F(t,y)|^2\,dy\frac{dt}{t}.
\end{array}$$
This proves \eqref{T22_equiv}. If now $F$ is supported in $T(B)$, then for every $x\notin B$, $\mathcal{A}F(x)=0$. Therefore,

$$\int_B|\mathcal{A}F|^2=\int_M |\mathcal{A}F|^2.$$
Then, by \eqref{T22_equiv},

$$\int_B|\mathcal{A}F|^2\simeq \int_0^\infty\int_M |F(t,y)|^2\,dy\frac{dt}{t}=\int_{T(B)}|F(t,y)|^2\,dy\frac{dt}{t},$$
which proves \eqref{tent_sup2}. If $p>2$, then using successively Minkowski's integral inequality and H\"{o}lder's inequality,

$$\begin{array}{rcl}

\left(\int_B|\mathcal{A}F|^p\right)^{2/p}&\leq &\left(\int_B dx\left(\int_0^\infty\frac{dt}{t}\frac{1}{V(x,t)}\int_{B(x,t)}|F(t,y)|^2\,dy\right)^{p/2}\right)^{2/p}\\\\
&\leq& \int_0^\infty\frac{dt}{t} \left(\int_B \left(\frac{1}{V(x,t)}\int_{B(x,t)}|F(t,y)|^2\,dy\right)^{p/2}\right)^{2/p}\\\\
&\leq&  \int_0^\infty\frac{dt}{t} \left(\int_B \frac{1}{V(x,t)}\int_{B(x,t)}|F(t,y)|^p\,dy\right)^{2/p}
\end{array}$$
If $(t,y)\in T(B)$, then

$$\{x\in B\,:\,y\in B(x,t)\}=B(y,t),$$
and therefore, using Lemma \ref{equiv_vol} and the fact that $F$ is supported in $T(B)$, one has

$$\begin{array}{rcl}
\left(\int_B|\mathcal{A}F|^p\right)^{2/p}&\leq & \int_0^r\frac{dt}{t} \left(\int_{\{y\,:\,(t,y)\in T(B)\}} |F(t,y)|^p \frac{V(y,t)}{V(x,t)}\,dy\right)^{2/p}\\\\
&\leq& C\int_0^r\frac{dt}{t} \left(\int_{\{y\,:\,(t,y)\in T(B)\}} |F(t,y)|^p \,dy\right)^{2/p}
\end{array}$$
But it is clear that $\{y\,:\,(t,y)\in T(B)\}\subset B$, and therefore

$$\left(\int_B|\mathcal{A}F|^p\right)^{2/p}\leq C\int_0^r\frac{dt}{t} \left(\int_{B} |F(t,y)|^p\,dy \right)^{2/p}.$$

\end{proof}
We now recall a result from \cite{CMS}, about the boundedness on $T^{2,p}$ of some maximal function: define the maximal function $\mathcal{C}$ by

$$\mathcal{C}F(x)=\sup_{x\in T(B)} \left(\int_{T(B)} |F(t,y)|^2 dy\frac{dt}{t}\right)^{1/2},$$
where $T(B)$ is the tent over the geodesic ball $B$. Then for $p>2$, the maximal function $\mathcal{C}$ is bounded on $T^{2,p}$: 

\begin{Thm}[\cite{CMS}, Theorem 3] \label{CMS}
Assume \eqref{D}. Then for every $2<p<\infty$, there exists a constant $C_p>0$ such that, for every $F\in T^{2,p}$,

$$C_p^{-1}||\mathcal{C}F||_{L^p}\leq||F||_{T^{2,p}}\leq C_p ||\mathcal{C}F||_{L^p}.$$
Also, there is a constant $C$ such that for every $x\in M$, and every $F$,

\begin{equation}\label{tent_max}
\mathcal{C}F(x)\leq C\left(\mathcal{M}\left(|\mathcal{A}F|^2\right)(x)\right)^{1/2}.
\end{equation}

\end{Thm}
In \cite{CMS} this result is proved in the case of $M=\R^n$, and the proof extends to the homogeneous spaces setting. For the sake of completeness and since inequality \eqref{tent_max} will be of great use to us, we provide indications of a proof which works also in the homogeneous spaces setting.\\\\

\noindent \textit{Proof of Theorem \ref{CMS}:}\\\\
The inequality 

$$||\mathcal{A}F||_{L^p}\leq C_p ||\mathcal{C}F||_{L^p}$$
for $2<p<\infty$ is a consequence of the duality of the tent spaces, which extends to the case of homogeneous spaces (see Theorem 4.4 in \cite{AMR} and the proof of Theorem 3 in \cite{CMS}). The inequality 

$$C_p^{-1}||\mathcal{C}F||_{L^p}\leq||\mathcal{A}F||_{L^p}$$
follows at one from \eqref{tent_max} and the strong $(\frac{p}{2},\frac{p}{2})$ type of the Hardy-Littlewood maximal function. It remains to prove \eqref{tent_max}. 
By similar arguments than in Lemma \ref{tent_sup}, one has, for any ball $B$,

$$\int_{T(B)}|F(t,y)|^2dy\frac{dt}{t}\leq C\int_B|\mathcal{A}F|^2.$$
Therefore,

$$\frac{1}{|B|}\int_{T(B)}|F(t,y)|^2dy\frac{dt}{t}\leq C \inf_{x\in B} \mathcal{M}(|\mathcal{A}F|^2)(x).$$
Taking the supremum over the set of balls $B$ containing a fixed point $x\in M$ yields \eqref{tent_max}. 

\cqfd

\subsection{Hardy spaces}

Now we recall the definition of Hardy spaces, following \cite{AMR}. Denote by $D=d+d^\star$ the Dirac operator, and $\Delta_k=dd^\star+d^\star d$ the Laplacian acting on $k-$forms. For $k=0$, we will simply write $\Delta$ instead of $\Delta_0$. \\

First, let us consider the case $p=2$. Define $H_k^2=\overline{\mathcal{R}(D)\cap L^2(\Lambda^k T^\star M)}^{L^2},$ $H^2_{k,d}=\overline{\mathcal{R}(d)\cap L^2(\Lambda^k T^\star M)}^{L^2},$ $H^2_{k,d^\star}=\overline{\mathcal{R}(d^\star)\cap L^2(\Lambda^k T^\star M)}^{L^2}$. They are Hardy spaces of $k-$forms associated respectively to $D$, $d$ and $d^\star$. We have the orthogonal decomposition in $L^2(\Lambda^k T^\star M)$:

$$H^2_k=H^2_{k,d}\oplus_\perp  H^2_{k,d^\star},$$
and the Hodge decomposition

$$L^2(\Lambda^kT^\star M)=H^2_k\oplus_\perp \mathcal{H}^k,$$
where 

$$\mathcal{H}^k=\{\omega\in L^2(\Lambda^k T^\star M)\,:\,\Delta_k\omega=0\}=\ker_{L^2}(D)\cap L^2(\Lambda^k T^\star M)$$ 
is the set of $L^2$ harmonic $k-$forms. There is a description of these Hardy spaces in term of tent spaces, which we describe now. For $\theta \in(0,\frac{\pi}{2})$, set

$$\Sigma_{\theta^+}^0=\{z\in \mathbb{C} \,:\, |\arg(z)|<\theta\},$$
$$ \Sigma_{\theta}^0=\Sigma_{\theta^+}^0\cup (-\Sigma_{\theta^+}^0).$$
Denote by $H^\infty(\Sigma_\theta^0)$ the algebra of holomorphic functions on $\Sigma^0_\theta$. Given $\sigma,\tau>0$, define $\Psi_{\sigma,\tau}(\Sigma_\theta^0)$ to be the set of holomorphic functions $\psi\in H^\infty(\Sigma_\theta^0)$ which satisfy

$$|\psi(z)|\leq C\inf(|z|^\sigma,|z|^{-\tau})$$
for some $C>0$ and all $z\in \Sigma_\theta^0$. Finally, let $\Psi(\sigma_\theta^0)=\bigcup_{\Sigma,\tau>0}\Psi_{\sigma,\tau}(\Sigma_\theta^0)$. For $\psi\in \Psi(\Sigma_\theta^0)$, define

$$\left(\mathcal{Q}_\psi \omega\right)_t=\psi_t(D)\omega,$$
then $\mathcal{Q}_\psi \omega$ belongs to the tent space $T^{2,2}$, and

$$||\mathcal{Q}_\psi\omega||_{T^{2,2}}\simeq ||\omega||_2,$$
for every $\omega\in H^2_k$. More precisely,

$$\mathcal{Q}_\psi : H^2_k\rightarrow T^{2,2}$$
is an isomorphism, with inverse $\mathcal{S}_{\tilde{\psi}}$ defined by

$$\mathcal{S}_{\tilde{\psi}} H=\int_0^\infty \tilde{\psi}_t(D) H_t \frac{dt}{t},$$
where $\tilde{\psi}\in \Psi(\Sigma_\theta^0)$ satisfies $\int_0^\infty \tilde{\psi}(\pm t)\psi(\pm t)\frac{dt}{t}=1$. There is a similar result for $H_{k,d}^2$ and $H_{k,d^\star}^2$: for $\psi\in \Psi_{1,\tau}(\Sigma_\theta^0)$, define 

$$\left(\mathcal{Q}_{d,\psi} \omega\right)_t=td\,\varphi_t(D)\omega,$$
and

$$\left(\mathcal{Q}_{d^\star,\psi} \omega\right)_t=td^\star\,\varphi_t(D)\omega,$$
where $\varphi(z)=\frac{\psi(z)}{z}$. Then for every $\omega\in H_{k,d}^2$ (resp. $H_{k,d^\star}^2$), $\mathcal{Q}_{d^\star,\psi}\omega$ (resp. $\mathcal{Q}_{d,\psi}\omega$) belongs to $T^{2,2}$, and

$$||\mathcal{Q}_{d^\star,\psi}\omega||_{T^{2,2}}\simeq ||\omega||_2\mbox{ (resp. }||\mathcal{Q}_{d,\psi}\omega||_{T^{2,2}}\simeq ||\omega||_2).$$
Moreover, $\mathcal{Q}_{d^\star,\psi} : H_{k,d}\to T^{2,2}$ (resp. $\mathcal{Q}_{d,\psi} : H_{k,d^\star}\to T^{2,2}$) is an isomorphism, with inverse $\mathcal{S}_{d,\tilde{\psi}}$ (resp. $\mathcal{S}_{d^\star,\tilde{\psi}}$) given by

$$\mathcal{S}_{d,\tilde{\psi}}F=\int_0^\infty td\, \tilde{\phi}_t(D)F\frac{dt}{t}\mbox{ (resp. }\mathcal{S}_{d^\star,\tilde{\psi}}F=\int_0^\infty td^\star\, \tilde{\phi}_t(D)F\frac{dt}{t},$$
with $\tilde{\psi}(z)=\frac{\tilde{\phi}(z)}{z}$ satisfying $\int_0^\infty \tilde{\psi}(\pm t)\psi(\pm t)\frac{dt}{t}=1$. An interesting choice for us will be $\varphi(z)=z^{2M}e^{-z^2}$ for some integer $M$: for this choice of $\varphi$, the functionals $\mathcal{Q}_{d,\psi}$ and $\mathcal{Q}_{d^\star,\psi}$ writes

$$\left(\mathcal{Q}_{d,\psi} \omega\right)_t=td\,(t\Delta_k)^Me^{-t\Delta_k}\omega,$$
and

$$\left(\mathcal{Q}_{d^\star,\psi} \omega\right)_t=td^\star\,(t\Delta_k)^Me^{-t\Delta_k}\omega.$$

Now, let us turn to the definition of the Hardy spaces when $p\neq 2$. If $1<p<2$ (resp. $2<p<\infty$), for any $\psi\in\Psi_{1,\beta+1}(\Sigma_\theta^0)$ (resp. $\psi\in\Psi_{\beta,2}(\Sigma_\theta^0)$) the Hardy space $H^p_k$ is defined as the closure of 

$$\{\omega\in H^2_k\,:\,||\mathcal{Q}_{\psi}\omega||_{T^{2,p}}<\infty\}$$
under the norm  $||\mathcal{Q}_{\psi}\omega||_{T^{2,p}}$. It turns out that the above definition of $H^p_k$ is actually independant of the choice of $\psi$ in the considered class of holomorphic functions (see Definition 5.5 and 5.6 in \cite{AMR}). The Hardy spaces $H^p_{k,d}$ and $H^p_{k,d^\star}$ are defined similarly, replacing $\mathcal{Q}_{\psi}$ by $\mathcal{Q}_{d^\star,\psi}$ and $\mathcal{Q}_{d,\psi}$ respectively: for example, if $1<p<2$ (resp. $2<p<\infty$), for any $\psi\in\Psi_{1,\beta+1}(\Sigma_\theta^0)$ (resp. $\psi\in\Psi_{\beta,2}(\Sigma_\theta^0)$) the Hardy space $H^p_{k,d}$ is defined as the closure of 

$$\{\omega\in H^2_{k,d}\,:\,||\mathcal{Q}_{d^\star,\psi}\omega||_{T^{2,p}}<\infty\}$$
under the norm  $||\mathcal{Q}_{d^\star,\psi}\omega||_{T^{2,p}}$.

\section{Our results}

In this section, we explain our results in greater details. Our first result concerns the boundedness of a generalised Lusin area integral. Define $\mathcal{Q}$ by 

$$(\mathcal{Q}f)_t=t\nabla e^{-t^2\Delta}f.$$
Notice that $\mathcal{Q}=\mathcal{Q}_{d,\psi}$ with $\psi(z)=ze^{-z^2}$. It is a classical result that in $\R^n$, for any $1<p<\infty$, 

$$\mathcal{Q} : L^p\rightarrow T^{2,p}$$
is bounded (see \cite{St2}, p. 91). It is claimed in \cite{AMR}, Remark 6.5 that the same result is true if the Hodge Laplacian on $1-$forms $\Delta_1$ satisfies Gaussian estimates, i.e. if \eqref{G_k} holds for $k=1$. The claim relies on an unpublished manuscript of Auscher, Duong and McIntosh. Furthermore, it is a consequence of \cite{AMR}, Corollary 6.3 that with the only assumption \eqref{D} on $M$, if $2\leq p<\infty$ and $\psi\in \Psi_{\beta,2}$ (recall that $\beta=[\frac{\kappa}{2}]+1$), then

$$\mathcal{Q}_{d,\psi} : L^p\rightarrow T^{2,p}$$
is bounded. Notice that in general, the function $\psi(z)=ze^{-z^2}$ does not belong to $\Psi_{\beta,2}$, and therefore Corollary 6.3 in \cite{AMR} does not say anything concerning the $L^p\rightarrow T^{2,p}$ boundedness of $\mathcal{Q}$. Our first result in this paper is that $L^p$ gradient estimates for the heat kernel are essentially enough to imply the $L^p\rightarrow T^{2,p}$ boundedness of $\mathcal{Q}$:

\begin{Thm}\label{area}

Let us assume that $M$ satisfies the doubling property \eqref{D}, the scaled Poincar\'{e} inequalities \eqref{P}, and for some $2<p<\infty$ the gradient estimate \eqref{grad_p} for the heat kernel. Then for every $2<q<p$, 

$$\mathcal{Q} : L^q\rightarrow T^{2,q}$$
is bounded.

\end{Thm}
The second result of this paper concerns the identification of some Hardy space with $L^p$, for $p\geq2$. The Hardy spaces that we will look at are the ones related to the Riesz transform $d\Delta^{-1/2}$, that is $H^p_{1,d}$. It is shown in \cite{AMR}, Theorem 5.15 that for any $1\leq p\leq \infty$,

$$d\Delta^{-1/2} : H^p_0=H^p_{0,d^\star}\rightarrow H^p_{1,d}$$
is an isomorphism. \\

If $M$ is a connected, complete, non-compact Riemannian manifold satisfying \eqref{D}, then for $1<p<\infty$, the range of $d^\star : \Lambda^1T^\star M\to \Lambda^0 T^\star M$ intersected with $L^p$, is equal to $L^p$: indeed if it is not the case, then there exists a non-zero $f\in L^{q}$, $\frac{1}{p}+\frac{1}{q}=1$ such that for any $\omega\in C_0^\infty(\Lambda^1 T^\star M)$,

$$\langle f, d^\star \omega\rangle =0,$$
which implies that $df=0$ in the weak (and thus, by elliptic regularity, in the strong) sense, therefore $f$ is constant and in $L^q$, and since by \eqref{D} $M$ has infinite volume, $f$ is identically zero, which is a contradiction. Thus, applying \cite{AMR}, Corollary 6.3, for any $2\leq p<\infty$,

$$L^p\subset H_0^p$$
and

$$\overline{\mathcal{R}(d)\cap L^p(\Lambda^1T^\star M)}^{L^p}\subset H^p_{1,d},$$
where $\mathcal{R}(d)$ denotes the range of $d$. We show in the second result of this paper that the $L^p$ gradient estimates of the heat kernel are essentially enough to prove that $H^p_{1,d}=\overline{\mathcal{R}(d)\cap L^p(\Lambda^1T^\star M)}^{L^p}$:

\begin{Thm}\label{Hardy}

Assume that $M$ is a complete Riemannian manifold satisfying the doubling property \eqref{D}, the scaled Poincar\'{e} inequalities \eqref{P} and for some $2<p<\infty$ the gradient estimates \eqref{grad_p} for the heat kernel. Then for every $2<q<p$,

$$H_{1,d}^q=\overline{\mathcal{R}(d)\cap L^q(\Lambda^1T^\star M)}^{L^q}.$$

\end{Thm}

\begin{Cor}\label{main2}{\em (= Main result)} 

Assume that $M$ is a connected, complete, non-compact Riemannian manifold satisfying the doubling property \eqref{D} and the scaled Poincar\'{e} inequalities \eqref{P}. Then for every $2<p_0<\infty$, the following are equivalent:

\begin{enumerate}

 \item For every $p\in (2,p_0)$, 
 
$$H_{1,d}^p=\overline{\mathcal{R}(d)\cap L^p(\Lambda^1T^\star M)}^{L^p}$$ with equivalent norms.
 
 \item For every $p\in (2,p_0)$, the Riesz transform $d\Delta^{-1/2}$ is bounded on $L^p$.

\end{enumerate}
If one these two equivalent conditions is satisfied, then for every $p\in (2,p_0)$,

$$||u||_p\simeq ||u||_{H^p}\simeq ||d\Delta^{-1/2}u||_p,\qquad \forall u\in C_0^\infty(M).$$

\end{Cor}

\begin{proof}

Assume that the equality 

$$H_{1,d}^p=\overline{\mathcal{R}(d)\cap L^p(\Lambda^1T^\star M)}^{L^p}$$
holds for every $p\in (2,p_0)$. Since, for every $2\leq p<\infty$, $L^p\subset H^p_0$ and the Riesz transform is bounded from $H^p_0$ to $H^p_{1,d}$, one gets that the Riesz transform is bounded on $L^p$ for every $p\in (2,p_0)$. This shows one implication of the equivalence claimed in Corollary \ref{main2}. For the converse implication: if we assume that the Riesz transform is bounded on $L^p$ for any $p\in (2,p_0)$, then (see \cite{ACDH}, it is the easy part of Theorem 1.3 therein) the gradient estimates for the heat kernel \eqref{grad_p} holds for any $p\in (2,p_0)$. Applying Theorem \ref{Hardy}, we get that the equality 

$$H_{1,d}^p=\overline{\mathcal{R}(d)\cap L^p(\Lambda^1T^\star M)}^{L^p}$$
holds for any $p\in (2,p_0)$. This conclude the proof.

\end{proof}

\begin{Rem}

{\em 

Theorem \ref{Hardy}, together with the the fact that the Riesz transform is $L^p\to H_{1,d}^p$ bounded, allows us to recover the result of Theorem 1.3 in \cite{ACDH}, according to which, in the class of manifolds satisfying \eqref{D} and \eqref{P}, the gradient estimates \eqref{grad_p} for the heat kernel implies the boundedness of the Riesz transform on $L^q$ for every $2<q<p$. Actually, our proof of Theorem \ref{Hardy} relies heavily on techniques developped in \cite{ACDH}.

}

\end{Rem}

\section{Proof of Theorem \ref{area}}

This section is devoted to the proof of Theorem \ref{area}. It relies on the following result, in the spirit of [\cite{ACDH}, ~Theorem 2.1], [\cite{AC}, ~Theorem 2.3] and [\cite{S}, Theorem 3.1]:

\begin{Pro}\label{bound_result}

Let $T=(T_t)_{t>0}$ be a linear operator, bounded from $L^2$ to $T^{2,2}$. Assume that there is $p>2$ and $\alpha>1$ such that for every ball $B$, and for every $f$,

\begin{equation}\label{improve1}
\left(\frac{1}{|B|}\int_B\mathcal{A}(\chi_{T(B)}Tg)^p\right)^{1/p}\leq C\left(\inf_{x\in B}\mathcal{C}(Tf)(x)+\inf_{x\in B}\left(\mathcal{M}\left( f^2\right)\right)^{1/2}(x)\right),
\end{equation}
where we have let $g:=(1-\chi_{\alpha B})f$. Then, for every $2<q<p$, $T$ is bounded from $L^q$ to $T^{2,q}$.

\end{Pro}
The proof of Proposition \ref{bound_result} is given in the Appendix. Let us now give the proof of Theorem \ref{area}:\\\\
\textit{Proof of Theorem \ref{area}:}\\\\
We state a preliminary result:

\begin{Pro}\label{off-diago}

Assume that the measure is doubling \eqref{D}, that the scaled Poincar\'{e} inequalities hold \eqref{P}, and the gradient estimate of the heat kernel \eqref{grad_p}. Then for every $2<q<p$, there is a constant $C$ such that for every $f$,

$$\left(\frac{1}{|B|}\int_B\mathcal{A}(\chi_{T(B)}\mathcal{Q}g)^q\right)^{1/q}\leq C\left(\inf_{x\in B}\mathcal{C}(\mathcal{Q}f)(x)+\inf_{x\in B}\left(\mathcal{M}\left( f^2\right)\right)^{1/2}(x)\right),$$
where we have let $g:=(1-\chi_{16 B})f$.

\end{Pro}
Let us assume for the moment the result of Proposition \ref{off-diago}. Proposition \ref{off-diago} implies that for the choice $T=\mathcal{Q}$, the hypotheses of Proposition \ref{bound_result} are satisfied. Therefore, by Proposition \ref{bound_result}, $\mathcal{Q}$ is bounded from $L^q$ to $T^{2,q}$, for every $2<q<p$. This concludes the proof of Theorem \ref{area}.

\cqfd

\begin{Rem}

{\em
Recall (Theorem \ref{CMS}) that the norm in $T^{2,q}$ is given by $||F||_{T^{2,q}}=||\mathcal{A}F||_q\simeq ||\mathcal{C}F||_q$. If the following stronger inequality were true:

\begin{equation}\label{improve2}
\left(\frac{1}{|B|}\int_B\mathcal{A}(\mathcal{Q}g)^p\right)^{1/p}\leq C\left(\inf_{x\in B}\mathcal{C}(\mathcal{Q}f)(x)+\inf_{x\in B}\left(\mathcal{M}\left( f^2\right)\right)^{1/2}(x)\right),
\end{equation}
where $g:=(1-\chi_{\alpha B})f$, then an application of Theorem 2.3 in \cite{AC} -- with the choice $T=\mathcal{A}\mathcal{Q}$ -- would yield directly Theorem \ref{area}. However, we have been unable to prove inequality \eqref{improve2}, and in fact we feel that \eqref{improve2} does not hold. Therefore, we have to adapt the argument of \cite{AC} in order to work with the weaker inequality \eqref{improve1}.

}

\end{Rem}

\noindent\textit{Proof of Proposition \ref{off-diago}:}\\\\
The proof is an elaboration on some arguments from \cite{ACDH}. Define $F=\mathcal{Q}g$. We will denote $F_t=F(t,\cdot)$. Fix a ball $B$ with radius $r$. By inequality \eqref{tent_sup_p} in Lemma \ref{tent_sup} applied to $\chi_{T(B)}F$, there holds:

$$\left(\frac{1}{|B|}\int_B\mathcal{A}(\chi_{T(B)}F)^q\right)^{2/q}\leq C\int_0^r\frac{dt}{t}\left(\frac{1}{|B|}\int_{B} |F(t,y)|^q\,dy\right)^{2/q}.$$
Arguments from \cite{ACDH} (more precisely, Lemma 3.2 and the proof of (3.12) therein), relying on $L^p$ off-diagonal estimates for the gradient of the heat semi-group, lead to the following lemma, whose proof is postponed:

\begin{Lem}\label{Lp_off_diago}
For $t\leq r$, the following inequality holds:


$$\left(\frac{1}{|B|}\int_{B} |F_t|^q\right)^{1/q}\leq C\left(\frac{r}{t}\right)\left(\frac{1}{|4B|}\Bint_{4B}|F_{t/2}|^2\right)^{1/2}+C \sum_{i=2}^\infty e^{-c2^i}\Lambda_{i,t}^{1/2},$$
where

$$\Lambda_{i,t}:=\sup_{l=2,\cdots, i+1}\frac{1}{|2^lB|}\Bint_{2^lB} |F_{t/2}|^2.$$

\end{Lem}
Assuming for the moment the result of Lemma \ref{Lp_off_diago}, let us finish the proof of Proposition \ref{off-diago}. By Lemma \ref{Lp_off_diago} and Minkowski's inequality, 

$$\begin{array}{rcl}
\left(\frac{1}{|B|}\Bint_B\mathcal{A}(\chi_{T(B)}F)^q\right)^{1/q}&\leq& C\left(\Bint_0^r \frac{dt}{t} \left(\frac{r}{t}\right)^2\frac{1}{|4B|}\Bint_{4B}|F_{t/2}|^2\right)^{1/2}\\\\
&&+ C\sum_{i=2}^\infty e^{-c2^i}\left(\Bint_0^r\frac{dt}{t}\Lambda_{i,t}\right)^{1/2}\\\\
&\leq&I+II
\end{array}$$
We claim that

\begin{equation}\label{I}
I\leq C\inf_{x\in B} \left(\mathcal{M}\left(f^2\right)(x)\right)^{1/2}
\end{equation}
and

\begin{equation}\label{II}
II\leq C\inf_{x\in B}\mathcal{C}(F)(x).
\end{equation}
The validity of the inequalities \eqref{I} and \eqref{II} implies at once that the result of Proposition \ref{off-diago} holds. In order that the proof of Proposition \ref{off-diago} be complete, it remains to prove inequalities \eqref{I} and \eqref{II}.\\\\
\underline{Step 1: estimate of $I$:}\\\\
It is a consequence of the $L^2$ off-diagonal estimates for $\mathcal{Q}$. Let us denote $C_1=16 B$, and for $i\geq2$, $C_i:=2^{i+3}B\setminus 2^{i+2}B$. Recall that $g=(1-\chi_{16B})f$, and write

$$g=\sum_{i=2}^\infty g\chi_{C_i}=\sum_{i=2}^\infty g_i,$$
where we have used the fact that $g$ is zero on $C_1=16 B$.  Then, using the $L^2$ off-diagonal estimates for $\mathcal{Q}_t=t\nabla e^{-t\Delta}$ (Lemma \ref{L2-off}), we get that for every $i\geq2$,

$$\int_{4 B}|\mathcal{Q}_{t/2} g_i|^2\leq C e^{-c4^i\left(\frac{r}{t}\right)^2}\int_{C_i}f^2.$$
Therefore, using \eqref{D},

$$\frac{1}{|4B|}\int_{4B}|\mathcal{Q}_{t/2}(D) g_i|^2\leq C 2^ie^{-c4^i\left(\frac{r}{t}\right)^2}\left(\frac{1}{|2^{i+3}B|}\int_{C_i}f^2\right).$$
By definition of the maximal function $\mathcal{M}$,

$$\frac{1}{|2^{i+3}B|}\int_{C_i}f^2\leq \frac{1}{|2^{i+3}B|}\int_{2^{i+3}B}f^2\leq \inf_{x\in B}\mathcal{M}(f^2)(x).$$
Consequently, 

$$\begin{array}{rcl}
I&\leq& \sum_{i=2}^\infty \left(\Bint_0^r \frac{dt}{t} \left(\frac{r}{t}\right)^2\frac{1}{|4B|}\Bint_{4B}|\mathcal{Q}_{t/2}g_i|^2\right)^{1/2}\\\\
&\leq&\sum_{i=2}^\infty\left( \Bint_0^r 2^i\left(\frac{r}{t}\right)^2 e^{-c4^i\left(\frac{r}{t}\right)^2} \frac{dt}{t}\right)^{1/2} \left(\inf_{x\in B}\left(\mathcal{M}(f^2)(x)\right)^{1/2}\right).
\end{array}$$
Next, for any $t\leq r$,

$$2^i\left(\frac{r}{t}\right)^2 e^{-4^i\left(\frac{r}{t}\right)^2}\leq C e^{-c\left(\frac{r}{t}\right)^2}e^{-c2^i\left(\frac{r}{t}\right)^2}\leq C e^{-c\left(\frac{r}{t}\right)^2}e^{-c2^i},$$
and thus, performing a change of variable,

$$ \begin{array}{rcl}
\Bint_0^r 2^i\left(\frac{r}{t}\right)^2 e^{-4^i\left(\frac{r}{t}\right)^2} \frac{dt}{t}&\leq& Ce^{-c2^i}\Bint_0^1e^{-\frac{c}{u^2}}\frac{du}{u}\\\\
&\leq& Ce^{-c2^i}.
\end{array}$$
Therefore,

$$I\leq C\inf_{x\in B}\left(\mathcal{M}(f^2)(x)\right)^{1/2},$$
which proves \eqref{I}.\\\\
\underline{Step 2: estimate of $II$:}\\\\
Performing a change of variables, we have
$$II\leq C\sum_{i=2}^\infty e^{-c2^i}\left(\int_0^{r/2}\frac{du}{u}\Lambda_{i,2u}\right)^{1/2}.$$
If $t\leq r$ and $y\in 2^{l}B$, then $(t,y)\in T(2.2^{l}B)$. Furthermore, by definition of the maximal function $\mathcal{C}$, we have 

$$\left(\frac{1}{|2.2^{l}B|}\int_{T(2.2^{l}B)} |F(t,y)|^2\frac{dt}{t} dy\right)^{1/2}\leq \inf_{x\in B}\mathcal{C}(F)(x),$$
and therefore, by definition of $\Lambda_{i,t}$ and using (D), for every $i\geq2$, there holds

$$\left(\int_0^{r/2}\frac{du}{u}\Lambda_{i,2u}\right)^{1/2}\leq C\inf_{x\in B}\mathcal{C}(F)(x).$$
Finally,
$$II\leq C\left(\sum_{i=2}^\infty e^{-c2^i} \right)\inf_{x\in B}\mathcal{C}(F)(x)=C'\inf_{x\in B}\mathcal{C}(F)(x).$$
This proves inequality \eqref{II}, and concludes the proof of Proposition \ref{off-diago}.
\cqfd

\noindent\textit{Proof of Lemma \ref{Lp_off_diago}:}\\\\
For the sake of completeness, we now give the proof of Lemma \ref{Lp_off_diago}. As we have already said, it is based on arguments from \cite{ACDH} (more precisely, from Lemma 3.2 and the proof of (3.12) therein). Define $\varphi(z)=e^{-\frac{1}{4}z^2}$, and write 

$$\begin{array}{rcl}
F_t&=&t\nabla e^{-\frac{3t^2}{4}\Delta} \varphi_t(D) g \\\\
&=&t\nabla e^{-\frac{3t^2}{4}\Delta} \left(\varphi_t(D) g- (\varphi_t(D) g)_{4B}\right)
\end{array}$$
where we have used the fact that under \eqref{D}, $e^{-t\Delta}\mathbf{1}=\mathbf{1}$. Let $h:=\varphi_t(D) g$. Notice that $h$ depends on $t$, but we do not write this dependance in order to keep the notation light. Define $C_1=4B$ and for $i\geq 2$, $C_i=2^{i+1}B\setminus 2^iB$. Let finally $h_i=(h-h_{4B})\chi_{C_i}$. We write

$$h-h_{4B}=\sum_{i=1}^\infty h_i.$$
With these notations and applying Minkowski's inequality, one has

\begin{equation}\label{sum_off_diago}
\left(\frac{1}{|B|}\int_{B} |F_t|^q\right)^{1/q}\leq \sum_{i=1}^\infty \left(\frac{1}{|B|}\int_B |t\nabla e^{-\frac{3t^2}{4}\Delta} h_i|^q\right)^{1/q}.\end{equation}
By inequalities \eqref{1_ultra_p} and \eqref{i_ultra_p}, there holds that for $i=1$,

\begin{equation}\label{L_p-g_1}
\left(\frac{1}{|B|}\int_B |t\nabla e^{-\frac{3t^2}{4}\Delta} h_1|^q\right)^{1/q}\leq C \left(\frac{1}{|4B|}\int_{C_1}|h_1|^2\right)^{1/2},
\end{equation}
and for every $i\geq2$,

\begin{equation}\label{L_p-g_i}
\left(\frac{1}{|B|}\int_B |t\nabla e^{-\frac{3t^2}{4}\Delta} h_i|^q\right)^{1/q}\leq C e^{-c4^i\left(\frac{r}{t}\right)^2}\left(\frac{1}{|2^{i+1}B|}\int_{C_i}|h_i|^2\right)^{1/2}.
\end{equation}
But by definition,

$$\int_{C_i} |h_i|^2\leq \int_{2^{i+1}B}|h-h_{4B}|^2.$$
Next, write

$$|h-h_{4B}|\leq |h-h_{2^{i+1}B}|+\sum_{l=2}^i|h_{2^lB}-h_{2^{l+1}B}|.$$
Therefore, by Minkowski's inequality,

$$\begin{array}{rcl}
\left(\frac{1}{|2^{i+1}B|}\Bint_{C_i} |h_i|^2\right)^{1/2}&\leq& \left(\frac{1}{|2^{i+1}B|}\Bint_{2^{i+1}B}|h-h_{2^{i+1}B}|^2\right)^{1/2}\\\\
&&+\sum_{l=2}^i|h_{2^lB}-h_{2^{l+1}B}|.
\end{array}$$
Applying the Poincar\'{e} inequality \eqref{P} on $2^{i+1}B$, we obtain

$$\left(\frac{1}{|2^{i+1}B|}\int_{2^{i+1}B}|h-h_{2^{i+1}B}|^2\right)^{1/2}\leq C (2^{i+1}r) \left(\frac{1}{|2^{i+1}B|}\int_{2^{i+1}B}|\nabla h|^2\right)^{1/2}.$$
Also, observe that, by Cauchy-Schwarz and the Poincar\'{e} inequality \eqref{P} on $2^{l+1}B$, we have

$$|h_{2^lB}-h_{2^{l+1}B}|\leq \left(\frac{1}{|2^{l}B|}\int_{2^{l}B}|h-h_{2^{l+1}B}|^2\right)^{1/2}\leq C(2^lr)\left(\frac{1}{|2^{l}B|}\int_{2^{l}B}|\nabla h|^2\right)^{1/2}.$$
As a consequence, for every $i\geq1$,

$$\left(\frac{1}{|2^{i+1}B|}\int_{C_i} |h_i|^2\right)^{1/2}\leq C \sum_{l=2}^{i+1} 2^lr\left(\frac{1}{|2^lB|}\int_{2^lB} |\nabla h|^2\right)^{1/2}.$$
Now, notice that 

$$\frac{t}{2}\nabla h=\frac{t}{2}\nabla\varphi_t(D)g=F_{t/2},$$
so that for every $i\geq1$,

\begin{equation}\label{g_i}
\left(\frac{1}{|2^{i+1}B|}\int_{C_i} |h_i|^2\right)^{1/2}\leq C \sum_{l=2}^{i+1} 2^l\left(\frac{r}{t}\right)\left(\frac{1}{|2^lB|}\int_{2^lB} |F_{t/2}|^2\right)^{1/2}.
\end{equation}
From \eqref{sum_off_diago}, \eqref{L_p-g_1}, \eqref{L_p-g_i} and \eqref{g_i}, we get that

\begin{equation}
\begin{array}{rcl}
&&\left(\frac{1}{|B|}\int_{B} |F_t|^q\right)^{1/q}\\\\
&&\,\,\,\leq C\left(\frac{1}{|4B|}\Bint_{4B} |F_{t/2}|^2\right)^{1/2}\\\\
&&\,\,\,\,\,\,\,\,+C\sum_{i=2}^\infty e^{-c4^i\left(\frac{r}{t}\right)^2}\sum_{l=2}^{i+1} 2^l\left(\frac{r}{t}\right)\left(\frac{1}{|2^lB|}\Bint_{2^lB} |F_{t/2}|^2\right)^{1/2}.
\end{array}
\end{equation}
Therefore, by definition of $\Lambda_{i,t}$,

$$\left(\frac{1}{|B|}\int_{B} |F_t|^q\right)^{1/q}\leq C\left(\frac{1}{|4B|}\int_{4B} |F_{t/2}|^2\right)^{1/2}+C\sum_{i=2}^\infty i2^{i+1}\left(\frac{r}{t}\right)e^{-c4^i\left(\frac{r}{t}\right)^2}\Lambda_{i,t}.$$
But for $t\leq r$ and $i\geq 1$, one has the elementary inequality

$$i2^{i+1}\left(\frac{r}{t}\right)e^{-c4^i\left(\frac{r}{t}\right)^2}\leq Ce^{-\frac{c}{2}\left(\frac{r}{t}\right)^2}.$$
Thus, for $t\leq r$,

$$\left(\frac{1}{|B|}\int_{B} |F_t|^q\right)^{1/q}\leq C\left(\frac{1}{|4B|}\Bint_{4B} |F_{t/2}|^2\right)^{1/2}+C\sum_{i=2}^\infty e^{-c4^i\left(\frac{r}{t}\right)^2}\Lambda_{i,t},$$
which is precisely the result of Lemma \ref{Lp_off_diago}.

\cqfd

\section{Proof of Theorem \ref{Hardy}}
In this section, we prove Theorem \ref{Hardy}. Let $\mathcal{S}$ be defined by

$$\mathcal{S}H=\int_0^\infty t\nabla e^{-t^2\Delta}H_t\frac{dt}{t}.$$
Define also 

$$\mathcal{Q}_N=\mathcal{Q}_{d^\star,\psi},$$
where $\psi(z)=z^{2N+1}e^{-z^2}$. That is,

$$\mathcal{Q}_N\omega=(t^2\Delta)^Ne^{-t^2\Delta}(td^\star \omega).$$
Theorem \ref{Hardy} is consequence of the following result:

\begin{Thm}\label{main_est}

Assume that $N\geq \frac{3}{4}\kappa$. Then for every $2<q<p$, there exists a constant $C$ such that for every $\omega\in H^2_{1,d}\cap H^q_{1,d}$,

$$||\mathcal{S}\mathcal{Q}_N \omega||_{L^q}\leq C ||\omega||_{H^q_{1,d}}.$$
\end{Thm}
\noindent Assuming for the moment the result of Theorem \ref{main_est}, let us give the proof of Theorem \ref{Hardy}:\\\\

\noindent{\em Proof of Theorem \ref{Hardy}:}\\\\

\noindent Fix $N\geq \frac{3}{4}\kappa$. Let us define $c>0$ by

$$c=\int_0^\infty t^{2N+2}e^{-t^2}\frac{dt}{t}.$$
Using the fact that for $\omega\in H^2_{1,d}$, 

$$\Delta_1\omega=dd^\star \omega,$$
we get by the Spectral Theorem that for every $\omega\in H^2_{1,d}$, there holds:

$$c^{-1}\mathcal{S}\mathcal{Q}_N\omega=\omega.$$
Therefore, applying Theorem \ref{main_est}, we get for $\omega\in H^2_{1,d}\cap H^q_{1,d}$ that

$$\begin{array}{rcl}
||\omega||_{L^q}&=&c^{-1}||\mathcal{S}\mathcal{Q}_N\omega||_{L^q}\\\\
&\leq& C||\omega||_{H^q_{1,d}}.
\end{array}$$
As a consequence, $H^2_{1,d}\cap H_{1,d}^q\subset\mathcal{R}(d)\cap L^q(\Lambda^1T^\star M)$. Since $H^2_{1,d}\cap H_{1,d}^q$ is dense in $H^q_{1,d}$, we get that

$$H_{1,d}^q\subset\overline{\mathcal{R}(d)\cap L^q(\Lambda^1T^\star M)}^{L^q}.$$
But by [\cite{AMR}, Corollary 6.3], the reverse inclusion $\overline{\mathcal{R}(d)\cap L^q(\Lambda^1T^\star M)}^{L^q}\subset H_{1,d}^q$ always holds. Thus,

$$H_{1,d}^q=\overline{\mathcal{R}(d)\cap L^q(\Lambda^1T^\star M)}^{L^q}.$$
This conclude the proof of Theorem \ref{Hardy}.

\cqfd
In the remaining part of this section, we prove Theorem \ref{main_est}.\\\\
\textit{Strategy of the proof of Theorem \ref{main_est}:}\\\\
We will denote $F=\mathcal{Q}_N\omega$. Our proof is inspired by the proof of [\cite{ACDH}, Theorem 1.3]. Let us explain roughly the strategy. Define the ``regularizing operator'' $A_r$ by

$$A_r=I-(I-e^{-r^2\Delta})^n,$$
where $n$ is an integer which will be chosen big enough later. We will show that for some ``maximal function" $G_\gamma$ (to be specified later), the following pair of inequalities holds for any ball $B$ of radius $r$:

$$\left(\frac{1}{|B|}\int_B |\mathcal{S}(I-A_r)F|^2\right)^{1/2}\leq C\inf_{x\in B} G_\gamma F(x)$$
and

$$\left(\frac{1}{|B|}\int_B |\mathcal{S}A_rF|^q\right)^{1/q}\leq C\inf_{x\in B} \left(\mathcal{M}\left(|\mathcal{S}F|^2\right)(x)\right)^{1/2}.$$
Her, $\mathcal{M}$ is the Hardy-Littlewood maximal function. Once this is done, the proof of Theorem \ref{main_est} follows by arguments similar to [\cite{ACDH}, Theorem 2.1] or [\cite{AC}, Proposition 3.2 and Corollary 3.3]. The rest of this section will be devoted to make the above arguments more precise. First, we define the maximal function $G_\gamma$. For $0\leq j\leq n$, define

$$F^j:=(t^2\Delta)^{N+j}e^{-t^2\Delta}(td^\star \omega).$$
Then we let

$$\mathcal{G}_\gamma H=\left(\mathcal{M}\left(|\mathcal{A}_\gamma H|^2\right)\right)^{1/2},$$
 where we recall that $\mathcal{A}_\gamma$ is the function corresponding to $\mathcal{A}$ when the aperture of the cone is chosen to be $\gamma>0$. Finally, we define

$$G_\gamma F=\sum_{0\leq j\leq n} \mathcal{G}_\gamma F^j.$$
With this settled, we prove:

\begin{Pro}\label{maxi_est}
Assume \eqref{grad_p}, $n\geq \frac{3}{4}\kappa$ and $N\geq n$. Then there is $\gamma>0$ such that for every $2<q<p$ and for every ball $B$ of radius $r$, there holds:

\begin{equation}\label{max_2}
\left(\frac{1}{|B|}\int_B |\mathcal{S}(I-A_r)F|^2\right)^{1/2}\leq C\inf_{x\in B} G_\gamma F(x)
\end{equation}
and

\begin{equation}\label{max_q}
\left(\frac{1}{|B|}\int_B |\mathcal{S}A_rF|^q\right)^{1/q}\leq C\inf_{x\in B} \left(\mathcal{M}\left(|\mathcal{S}F|^2\right)(x)\right)^{1/2}.
\end{equation}

\end{Pro}
Let us first show how Theorem \ref{main_est} follows, assuming the result of Proposition \ref{maxi_est}:\\

\noindent\textit{End of the proof of Theorem \ref{main_est}:}\\\\
\noindent Let us fix $2<q<p$. Proposition \ref{maxi_est} allows us to apply Proposition 3.2 and Corollary 3.3 from \cite{AC} with the choice $\mathbf{F}=|\mathcal{S}F|^2$, $\mathbf{G}_B=2|\mathcal{S}(I-A_r)F|^2$, $\mathbf{H}_B=2|\mathcal{S}A_rF|^2$ and $G=G_\gamma F $, in order to obtain

$$||\mathbf{F}||_{\frac{q}{2}}\leq ||\mathcal{M}\mathbf{F}||_{\frac{q}{2}}\leq C||G||_{\frac{q}{2}}.$$
By definition, 

$$||\mathbf{F}||_{\frac{q}{2}}=||\mathcal{S}\mathcal{Q}_N\omega||^2_{q}.$$
Using the strong $(\frac{q}{2},\frac{q}{2})$ type of the maximal operator $\mathcal{M}$, we have

$$\begin{array}{rcl}
||G||_{\frac{q}{2}}&\leq& \sum_{0\leq j\leq n} ||\mathcal{M} |\mathcal{A}_\gamma F^j|^2||_{\frac{q}{2}}\\\\
&\leq& C\sum_{0\leq j\leq n} ||\mathcal{A}_\gamma F^j||^2_{q}\\\\
&\leq& C\sum_{0\leq j\leq n}  ||F^j||_{T^{2,q}}^2,
\end{array}$$
where we have used the fact that under \eqref{D}, different angles in the definition of the tent space $T^{2,q}$ give rise to equivalent norms (see Proposition 4 in \cite{CMS}). Recall that $F^j=(t^2\Delta)^{N+j}e^{-t^2\Delta}(td^\star \omega)$. According to the results of \cite{AMR} (see Definition 5.6 therein), for every $M\geq \frac{\beta}{2}$ (recall that $\beta=[\frac{\kappa}{2}]+1$),

$$||\omega||_{H^q_{1,d}}\simeq ||(t^2\Delta)^Me^{-t^2\Delta}(td^\star \omega)||_{T^{2,q}}.$$
Since for every $0\leq j\leq n$, there holds that $N+j\geq \frac{3}{4}\kappa\geq \frac{\beta}{2}$, we get by definition of $F^j$ that 

$$\sum_{0\leq j\leq n}  ||F^j||_{T^{2,q}}^2\leq C||\omega||^2_{H^q_{1,d}}.$$
Therefore,

$$||\mathcal{S}\mathcal{Q}_N\omega||_{q}\leq C ||\omega||_{H^q_{1,d}}.$$
The proof of Theorem \ref{main_est} is complete.

\cqfd

\noindent\textit{Proof of Proposition \ref{maxi_est}:}\\\\
\noindent Let us begin by establishing \eqref{max_q}, which follows directly from results in \cite{ACDH}. Notice that $A_r$ is a sum of terms of the form $C_k e^{-kr^2\Delta}$, for $1\leq k\leq n$. It is thus enough to prove that for any $1\leq k\leq n$,

\begin{equation}\label{eq_q}
\left(\frac{1}{|B|}\int_B |\mathcal{S}e^{-kr^2\Delta}F|^q\right)^{1/q}\leq C\inf_{x\in B} \left(\mathcal{M}\left(|\mathcal{S}F|^2\right)(x)\right)^{1/2}.
\end{equation}
By definition of $\mathcal{S}$,

$$\left(\frac{1}{|B|}\int_B |\mathcal{S}e^{-kr^2\Delta}F|^q\right)^{1/q}=\left(\frac{1}{|B|}\int_B |\nabla e^{-kr^2\Delta}g|^q\right)^{1/q},$$
where

$$g:=\int_0^\infty te^{-t^2\Delta}F_t\frac{dt}{t}.$$
According to Equation (3.12) in \cite{ACDH}, there holds:

$$\left(\frac{1}{|B|}\int_B |\nabla e^{-kr^2\Delta}g|^q\right)^{1/q}\leq C \inf_{x\in B}\left(\mathcal{M}\left(|\nabla g|^2\right)(x)\right)^{1/2}.$$
But

$$\nabla g=\int_0^\infty t\nabla e^{-t^2\Delta}F_t\frac{dt}{t}=\mathcal{S}F,$$
hence \eqref{eq_q}. \\\\

Now, we will be concerned with the more difficult task of establishing \eqref{max_2}. Again, this relies on ideas developped in \cite{ACDH}. In the proof, $C$ and $c$ will design generic constants, whose value can change from a line to another.  Expanding the term $I-A_r=(I-e^{-r^2\Delta})^n$ and performing a change of variable for each term appearing in the sum, we get

\begin{equation}\label{dvp_sum}
\begin{array}{rcl}
\mathcal{S}(I-A_r)F&=&\Bint_0^\infty \left(\sum_{k=0}^n (-1)^k \binom{n}{k} \chi_{\{t>\sqrt{k}r\}}\,t\nabla e^{-t^2\Delta} \frac{F_{\sqrt{t^2-kr^2}}}{\sqrt{t^2-kr^2}}\right)\,dt\\\\
&=&\mathcal{S}_1+\mathcal{S}_2,
\end{array}
\end{equation}
where $\mathcal{S}_1$ (resp. $\mathcal{S}_2$) corresponds to the integral being taken from $0$ to $Ar$ (resp. from $Ar$ to $\infty$), where the value of $A\geq \sqrt{n+1}$ will be precised later. We will separately estimate $\left(\frac{1}{|B|}\int_B|\mathcal{S}_1|^2\right)^{1/2}$ and $\left(\frac{1}{|B|}\int_B|\mathcal{S}_2|^2\right)^{1/2}$.\\\\

\noindent \underline{Step 1:} estimate of $\left(\frac{1}{|B|}\int_B|\mathcal{S}_1|^2\right)^{1/2}$ :\\\\
We claim that

\begin{Lem}\label{S_1}

The following inequality takes place:

$$\left(\frac{1}{|B|}\int_B|\mathcal{S}_1|^2\right)^{1/2}\leq C \inf_{x\in B} \mathcal{C}F(x).$$

\end{Lem}
The proof of Lemma \ref{S_1} is technical, and is postponed until the end of this section.\\\\

\noindent \underline{Step 2:} estimate of $\left(\frac{1}{|B|}\int_B|\mathcal{S}_2|^2\right)^{1/2}$ :\\\\
Let us define

$$f(s)=t\nabla e^{-t^2\Delta}\frac{F_{\sqrt{t^2-s}}}{\sqrt{t^2-s}}.$$
Then, $f(kr^2)=t\nabla e^{-t^2\Delta}\frac{F_{\sqrt{t^2-kr^2}}}{\sqrt{t^2-kr^2}}$ is the term appearing in the integrand of \eqref{dvp_sum}. We have, by the Taylor formula for $f$ around $0$:

\begin{equation}\label{Taylor}
f(kr^2)=\sum_{l=0}^{n-1}\frac{f^{(l)}(0)}{l!}k^lr^{2l}+\frac{k^nr^{2n}}{(n-1)!}\int_0^1(1-u)^{n-1}f^{(n)}(ukr^2)\,du
\end{equation}
Replacing $f(kr^2)$ by its expression given by \eqref{Taylor} in the sum $\sum_{k=0}^n(-1)^k\binom{n}{k}f(kr^2)$, and using the fact that for every integer $0\leq l\leq n-1$,

$$\sum_{k=0}^n(-1)^k\binom{n}{k} k^l=0,$$
(see [\cite{F}, Problem 16, p. 65]), we obtain

$$\left(\frac{1}{|B|}\int_B|\mathcal{S}_2|^2\right)^{1/2}\leq C\int_{Ar}^\infty r^{2n}\sup_{s\in [0, nr^2]}\left(\frac{1}{|B|}\int_B|f^{(n)}(s)|^2\right)^{1/2}\,dt.$$
Let us now compute $f^{(n)}(s)$: recalling that $F_t=(t^2\Delta)^Ne^{-t^2\Delta} td^\star \omega$, we have

$$f(s)=(t^2-s)^{-1/2}t\nabla e^{-t^2\Delta} \left((t^2-s)\Delta\right)^Ne^{-(t^2-s)\Delta} d^\star \omega,$$
and thus, assuming that $N>n$,

\begin{equation}\label{derivee}
f^{(n)}(s)=\frac{1}{(t^2-s)^{n+\frac{1}{2}}}\sum_{j\leq n} C(j,n,N) t\nabla e^{-t^2\Delta} F^j_{\sqrt{t^2-s}},
\end{equation}
where we recall that by definition, $F^j_t=(t^2\Delta)^{N+j}e^{-t^2\Delta} td^\star \omega$. We claim:

\begin{Lem}\label{est_max_j}

There exists $\gamma>1$ such that for every $H\in T^{2,2}$,

$$
\int_{Ar}^\infty r^{2n}\sup_{s\in [0, nr^2]}\frac{1}{|B|^{1/2}}\left|\left|\frac{1}{(t^2-s)^{n+\frac{1}{2}}}t\nabla e^{-t^2\Delta} H_{\sqrt{t^2-s}}\right|\right|_{L^2(B)}dt\leq C \inf_{x\in B} \mathcal{G}_\gamma H(x),
$$
(recall that $\mathcal{G}_\gamma H:=\left(\mathcal{M}\left(|\mathcal{A}_\gamma H|^2\right)\right)^{1/2}$).

\end{Lem}
The proof of Lemma \ref{est_max_j} is given below. Assuming it for the moment, and applying it to $H=F^j$ for every $0\leq j\leq n$, we obtain that by definition of $G_\gamma$,

\begin{equation}\label{S_2}
\left(\frac{1}{|B|}\int_B |\mathcal{S}_2|^2\right)^{1/2}\leq C\inf_{x\in B}\sum \mathcal{G}_\gamma F^j=C\inf_{x\in B}G_\gamma F(x).
\end{equation}
From Lemma \ref{S_1} and inequality \eqref{tent_max}, we get

\begin{equation}\label{S_12}
\left(\frac{1}{|B|}\int_B |\mathcal{S}_1|^2\right)^{1/2}\leq C\inf_{x\in B}\left(\mathcal{M}\left(|\mathcal{A}F|^2\right)(x)\right)^{1/2}.
\end{equation}
As consequence of \eqref{D}, since $\gamma>1$, there is a constant $C$ such that for every $F$,

$$\mathcal{A}F\leq C\mathcal{A}_\gamma F,$$
By \eqref{S_2} and \eqref{S_12}, we thus obtain \eqref{max_2}, and the proof of Proposition \ref{maxi_est} is complete.

\cqfd

\noindent\textit{Proof of Lemma \ref{S_1}:}\\

\noindent Define $C_1=4B$, and for $i\geq2$, $C_i=2^{i+1}B\setminus 2^iB$. Decomposing

$$F_t=\sum_{i=0}^\infty \chi_{C_i}F_t=\sum_{i=0}^\infty F_{i,t},$$
we have by Minkowski's inequality (with an obvious definition of $\mathcal{S}_{1,i}$ in which $F_i$ appears instead of $F$)

$$\left(\frac{1}{|B|}\int_B|\mathcal{S}_1|^2\right)^{1/2}\leq \sum_{i=1}^\infty\left(\frac{1}{|B|}\int_B|\mathcal{S}_{1,i}|^2\right)^{1/2}.$$
Our task is thus to estimate each  term of this sum. We first consider the case $i=1$. First, notice that

$$\left(\frac{1}{|B|}\int_B|\mathcal{S}_{1,i}|^2\right)^{1/2}\leq \left(\frac{1}{|B|}\int_B \left|\mathcal{S}(I-A_r)\left(\chi_{(0,Ar)\times 4B}F\right)\right|^2\right)^{1/2}.$$
It is easy to see that $(0,Ar)\times 4B\subset T((A+4)B)$, and thus

$$\left(\frac{1}{|B|}\int_B|\mathcal{S}_{1,i}|^2\right)^{1/2}\leq \left(\frac{1}{|B|}\int_B |\mathcal{S}(I-A_r)\left(\chi_{T((A+4)B)}F\right)|^2\right)^{1/2}.$$
We claim that $\mathcal{S}(I-A_r)$ is bounded from $T^{2,2}$ to $L^2$. Indeed, $\mathcal{S}$ is bounded from $T^{2,2}$ to $L^2$, and it is enough to see that $(I-A_r)$ is bounded on $T^{2,2}$. Since $I-A_r$ is a sum of terms of the form $C_ke^{-kr^2\Delta}$ for $1\leq k\leq n$, this follows from the fact that for any $1\leq k\leq n$, and every $H\in T^{2,2}$, using \eqref{T22_equiv},

$$\begin{array}{rcl}
||e^{-kr^2\Delta}H||_{T^{2,2}}&\simeq&\left(\Bint_0^\infty\Bint_M |e^{-kr^2\Delta} H_t(x)|^2\,dx\frac{dt}{t}\right)^{1/2}\\\\
&=&\left(\Bint_0^\infty ||e^{-kr^2\Delta}H_t||_2^2\frac{dt}{t}\right)^{1/2}\\\\
&\leq&\left(\Bint_0^\infty ||H_t||_2^2\frac{dt}{t}\right)^{1/2}\simeq ||H||_{T^{2,2}},
\end{array}$$
where we have used the fact that the semi-group $e^{-t\Delta}$ is contractive on $L^2$. Therefore,

$$||\mathcal{S}(I-A_r)\left(\chi_{T((A+4)B)}F\right)||_{L^2(B)}\leq C||\chi_{T((A+4)B)}F||_{T^{2,2}}.$$
Using \eqref{D}, we obtain by definition of the maximal function $\mathcal{C}$ that

$$\frac{1}{|B|^{1/2}}||\chi_{T((A+4)B)}F||_{T^{2,2}}\leq C \inf_{x\in B}\mathcal{C}F(x).$$
This implies that

$$\left(\frac{1}{|B|}\int_B|\mathcal{S}_{1,1}|^2\right)^{1/2}\leq C \inf_{x\in B}\mathcal{C}F(x).$$
Let us now turn to the case $i\geq 2$. By \eqref{dvp_sum} and Minkowski's integral inequality, we have

$$\left(\frac{1}{|B|}\int_B|\mathcal{S}_{1,i}|^2\right)^{1/2}\leq C\sum_{k=0}^n\int_{r\sqrt{k}}^{Ar}\left(\frac{1}{|B|}\int_B \left|t\nabla e^{-t^2\Delta} F_{i,{\sqrt{t^2-kr^2}}}\right|^2\right)^{1/2}\,\frac{dt}{\sqrt{t^2-kr^2}}.$$
Using the $L^2$ off-diagonal estimates for $t\nabla e^{-t^2\Delta}$ (Lemma \ref{L2-off}), we get

$$\begin{array}{rcl}
\left(\frac{1}{|B|}\Bint_B|\mathcal{S}_{1,i}|^2\right)^{1/2}&\leq& C\sum_{k=0}^n\Bint_{r\sqrt{k}}^{Ar} e^{-c4^i\left(\frac{r}{t}\right)^2} \frac{1}{|B|^{1/2}}\left|\left| F_{\sqrt{t^2-kr^2}}\right|\right|_{L^2(C_i)}\frac{dt}{\sqrt{t^2-kr^2}}\\\\
&\leq&C\Bint_0^{Ar}e^{-c4^i\left(\frac{r}{t}\right)^2}\frac{1}{|B|^{1/2}}\left|\left| F_t\right|\right|_{L^2(C_i)}\frac{dt}{t}\\\\
&&+Ce^{-c4^i}\sum_{k=1}^n\Bint_{r\sqrt{k}}^{Ar} \frac{1}{|B|^{1/2}}\left|\left| F_{\sqrt{t^2-kr^2}}\right|\right|_{L^2(C_i)}\frac{dt}{\sqrt{t^2-kr^2}}.
\end{array}$$
For $k\geq1$, we have by a change of variables and using the fact that $\frac{1}{t}\geq \frac{C}{r}$ on $(r,Ar)$:

$$\begin{array}{rcl}
\Bint_{r\sqrt{k}}^{Ar} \frac{1}{|B|^{1/2}}\left|\left| F_{\sqrt{t^2-kr^2}}\right|\right|_{L^2(C_i)}\frac{dt}{\sqrt{t^2-kr^2}}&\leq& \frac{C}{r}\Bint_0^{r\sqrt{A^2-k}}\frac{1}{|B|^{1/2}}||F_t||_{L^2(C_i)}dt\\\\
\leq\frac{C}{r}\left(\Bint_0^{Ar} tdt\right)^{1/2}&\times&\left(\Bint_0^{Ar}\frac{1}{|B|}||F_t||^2_{L^2(C_i)}\frac{dt}{t}\right)^{1/2}.
\end{array}$$
But it is easy to see that $C_i\times (0,Ar)\subset T\left((2^{i+1}+A)B\right)\subset T(A2^{i+1}B)$, and thus using (D),

\begin{equation}\label{est_CS}
\begin{array}{rcl}
\Bint_0^{Ar}\left(\frac{1}{|B|}||F_t||^2_{L^2(C_i)}\frac{dt}{t}\right)^{1/2}&\leq& \left(\frac{|A\,2^{i+1}B|}{|B|}\right)^{1/2}\\\\
&&\times\left(\frac{1}{|A2^{i+1}B|}\Bint_{T(A2^{i+1}B)}|F(t,x)|^2\,dx\frac{dt}{t}\right)^{1/2}\\\\
&\leq& C 2^{ i\kappa/2}\inf_{x\in B} \mathcal{C}F(x).
\end{array}
\end{equation}
Therefore,

$$\sum_{k=1}^n\Bint_{\sqrt{k}r}^{Ar} \frac{1}{|B|^{1/2}}\left|\left| F_{\sqrt{t^2-kr^2}}\right|\right|_{L^2(C_i)}\frac{dt}{\sqrt{t^2-kr^2}}\leq C 2^{i\kappa/2}\inf_{x\in B} \mathcal{C}F(x).$$
For $k=0$, we use successively Cauchy-Schwarz inequality, the inequality $e^{c4^i\left(\frac{r}{t}\right)^2}\leq e^{-C4^i}e^{-C4^i\left(\frac{r}{t}\right)^2}$ which holds for every $t\leq Ar$, and inequality \eqref{est_CS} in order to get:

$$\begin{array}{rcl}
\Bint_0^{Ar}e^{-c4^i\left(\frac{r}{t}\right)^2}\frac{1}{|B|^{1/2}}||F_t||_{L^2(C_i)}\frac{dt}{t}&\leq&\left(\int_0^{Ar}e^{-c4^i\left(\frac{r}{t}\right)^2}\frac{dt}{t}\right)^{1/2}\\\\
&&\times\left(\Bint_0^{Ar}\frac{1}{|B|}||F_t||^2_{L^2(C_i)}\frac{dt}{t}\right)^{1/2}\\\\
&\leq& C e^{-C4^i} 2^{i\kappa /2}\inf_{x\in B} \mathcal{C}F(x).
\end{array}$$
Finally, we obtain

$$\left(\frac{1}{|B|}\int_B|\mathcal{S}_{1,i}|^2\right)^{1/2}\leq Ce^{-C4^i}2^{i\kappa/2}\inf_{x\in B}\mathcal{C}F(x).$$
Summing on $i$, we get

$$\begin{array}{rcl}
\left(\frac{1}{|B|}\Bint_B|\mathcal{S}_1|^2\right)^{1/2}&\leq&\sum_{i=1}^\infty \left(\frac{1}{|B|}\Bint_B|\mathcal{S}_{1,i}|^2\right)^{1/2}\\\\
&\leq& C\left(\sum_{i=1}^\infty e^{-C4^i}2^{i\kappa/2}\right)\inf_{x\in B}\mathcal{C}F(x)\\\\
&\leq& C\inf_{x\in B}\mathcal{C}F(x).
\end{array}$$
This concludes the proof of Lemma \ref{S_1}.

\cqfd

\noindent\textit{Proof of Lemma \ref{est_max_j}:}\\

\noindent First, since $\frac{1}{t^2-nr^2}\leq C\frac{1}{t^2}$ for every $t>Ar$, we get

$$\begin{array}{rcl}
&\Bint_{Ar}^\infty& r^{2n}\sup_{s\in [0, nr^2]}\frac{1}{|B|^{1/2}}\left|\left|\frac{1}{(t^2-s)^{n+\frac{1}{2}}}t\nabla e^{-t^2\Delta} H_{\sqrt{t^2-s}}\right|\right|_{L^2(B)}dt\\\\
&\leq& C\Bint_{Ar}^\infty \left(\frac{r}{t}\right)^{2n} \sup_{s\in [0,nr^2]}\frac{1}{|B|^{1/2}}\left|\left|t\nabla e^{-t^2\Delta}H_{\sqrt{t^2-s}}\right|\right|_{L^2(B)}\frac{dt}{t}.
\end{array}$$
From now on, we fix $s\in [0,nr^2]$. Let $\lambda\in(0,1)$ be such that $A\lambda\geq 1$ (the value of $\lambda$ will be precised later). Denote $B_t:=\frac{t}{r} B=B(x_B,t)$, then for every $t\geq Ar$ and for every real function $\varphi$, there holds, using \eqref{D}, 

\begin{equation}\label{B_t}
\begin{array}{rcl}
\frac{1}{|B|^{1/2}}||\varphi||_{L^2(B)}&\leq& \left(\frac{|\lambda B_t|}{|B|}\right)^{1/2}\frac{1}{|\lambda B_t|^{1/2}}||\varphi||_{L^2(\lambda B_t)}\\\\
&\leq& C\left(\frac{t}{r}\right)^{\kappa/2}\frac{1}{|\lambda B_t|^{1/2}}||\varphi||_{L^2(\lambda B_t)}.
\end{array}
\end{equation}
Applying inequality \eqref{B_t} with $\varphi=t\nabla e^{-t^2\Delta}H_{\sqrt{t^2-s}}$, we obtain, denoting $\alpha=2n-\frac{\kappa}{2}$,

$$\begin{array}{rcl}
&\Bint_{Ar}^\infty& \left(\frac{r}{t}\right)^{2n} \sup_{s\in [0,nr^2]}\frac{1}{|B|^{1/2}}\left|\left|t\nabla e^{-t^2\Delta}H_{\sqrt{t^2-s}}\right|\right|_{L^2(B)}\frac{dt}{t}\\\\
&\leq& C\Bint_{Ar}^\infty\left(\frac{r}{t}\right)^\alpha\frac{1}{|\lambda B_t|^{1/2}}\left|\left|t\nabla e^{-t^2\Delta}H_{\sqrt{t^2-s}}\right|\right|_{L^2(\lambda B_t)}\frac{dt}{t}
\end{array}.$$
Denote $C^1_t=B(x_B,4t+4r)$ and for $i\geq 2$, 

$$C^i_t=B(x_B,4t+2^{i+1}r)\setminus B(x_B,4t+2^{i}r).$$
Notice that for every $i\geq 2$,

$$d(C_t^i,4B_t)=2^ir.$$
We decompose

$$H_t=\sum_{i=1}^\infty\chi_{C_t^i} H_t=\sum_{i=1}^\infty H_{i,t}.$$
We have by Minkowski's inequality

$$\begin{array}{rcl}
&\Bint_{Ar}^\infty&\left(\frac{r}{t}\right)^\alpha\frac{1}{|\lambda B_t|^{1/2}}\left|\left|t\nabla e^{-t^2\Delta}H_{\sqrt{t^2-s}}\right|\right|_{L^2(\lambda B_t)}\frac{dt}{t}\\\\
&\leq&\sum_{i=1}^\infty  \Bint_{Ar}^\infty\left(\frac{r}{t}\right)^\alpha\frac{1}{|\lambda B_t|^{1/2}}\left|\left|t\nabla e^{-t^2\Delta}H_{i,\sqrt{t^2-s}}\right|\right|_{L^2(\lambda B_t)}\frac{dt}{t}
\end{array}$$
It is enough to estimate each term in the above sum. We begin with the case $i=1$, for which the choice of $\lambda$ has no importance. Using the fact that for $t\geq r$, $H_{1,t}$ is supported in $C_t^1\subset 8B_t$ and the uniform boundedness on $L^2$ of $t\nabla e^{-t^2\Delta}$, we get

$$\begin{array}{rcl}
\left|\left|t \nabla e^{-t^2\Delta}H_{1,\sqrt{t^2-s}}\right|\right|_{L^2(\lambda B_t)}&\leq& \left|\left|t \nabla e^{-t^2\Delta}H_{1,\sqrt{t^2-s}}\right|\right|_2\\\\
&\leq &C \left|\left| H_{1,\sqrt{t^2-s}}\right|\right|_2=C \left|\left| H_{\sqrt{t^2-s}}\right|\right|_{L^2(8B_{\sqrt{t^2-s}})}.
\end{array}$$
By a change of variable and \eqref{D}, we obtain

$$\begin{array}{rcl}
&\Bint_{Ar}^{\infty}&\left(\frac{r}{t}\right)^\alpha\frac{1}{| \lambda B_t|^{1/2}}\left|\left|t\nabla e^{-t^2\Delta}H_{1,\sqrt{t^2-s}}\right|\right|_2\frac{dt}{t}\\\\
&\leq& C\Bint_{\sqrt{A^2r^2-s}}^\infty \left(\frac{r}{t}\right)^\alpha \frac{1}{|8B_t|^{1/2}}||H_t||_{L^2(8B_t)}\frac{dt}{t}\\\\
&\leq&C\Bint_r^\infty \left(\frac{r}{t}\right)^\alpha \frac{1}{|8B_t|^{1/2}}||H_t||_{L^2(8B_t)}\frac{dt}{t}
\end{array}$$
Applying Cauchy-Schwarz's inequality,

$$\begin{array}{rcl}
&\Bint_{Ar}^{\infty}&\left(\frac{r}{t}\right)^\alpha\frac{1}{| \lambda B_t|^{1/2}}\left|\left|t\nabla e^{-t^2\Delta}H_{1,\sqrt{t^2-s}}\right|\right|_2\frac{dt}{t}\\\\
&\leq& \left(\Bint_r^\infty \left(\frac{r}{t}\right)^{2\alpha}\frac{dt}{t}\right)^{1/2}\left(\Bint_r^\infty \frac{1}{|8B_t|}||H_t||^2_{L^2(8B_t)}\frac{dt}{t}\right)^{1/2}\\\\
&\leq &C\left(\Bint_r^\infty \frac{1}{|8B_t|}||H_t||^2_{L^2(8B_t)}\frac{dt}{t}\right)^{1/2}
\end{array}$$
We claim that 

\begin{equation}\label{est_max_0}
\int_r^\infty \frac{1}{|8B_t|}||H_t||^2_{L^2(8B_t)}\frac{dt}{t}\leq C\inf_{x\in B}\left(\mathcal{M}\left(|\mathcal{A}_{16}H|^2\right)(x)\right)^{1/2}.
\end{equation}
Indeed, let us compute

$$\frac{1}{|B|}\int_B |\mathcal{A}_{16}H|^2=\frac{1}{|B|}\int_B\int_0^\infty\frac{dt}{t}\frac{1}{V(x,16t)}\int_{B(x,16t)}|H|^2.$$
Notice that for every $x\in B$ and every $t>r$, $8B_t\subset B(x,16t)$. Also, by Lemma \ref{equiv_vol}, $V(x,16t)\simeq |8B_t|$. Therefore, 

$$\inf_{x\in B}\mathcal{M}\left(|\mathcal{A}_{16}H|^2\right)(x)\geq \frac{1}{|B|}\int_B |\mathcal{A}_{16}H|^2\geq C\int_r^\infty \frac{dt}{t}\frac{1}{|8B_t|}\int_{8B_t}|H(t,y)|^2\,dy,$$
which is what has been claimed. Consequently,

$$\int_{Ar}^\infty\left(\frac{r}{t}\right)^\alpha\frac{1}{| \lambda B_t|^{1/2}}\left|\left|t\nabla e^{-t^2\Delta}H_{1,\sqrt{t^2-s}}\right|\right|_{L^2(\lambda B_t)}\frac{dt}{t}\leq C \inf_{x\in B}\left(\mathcal{M}\left(|\mathcal{A}_{16}H|^2\right)(x)\right)^{1/2}.$$
Let us now turn to the case $i\geq2$. By a change of variables, we get

$$\begin{array}{rcl}
&\Bint_{Ar}^\infty&\left(\frac{r}{t}\right)^\alpha\frac{1}{|\lambda B_t|^{1/2}}\left|\left|t\nabla e^{-t^2\Delta}H_{i,\sqrt{t^2-s}}\right|\right|_{L^2(\lambda B_t)}\frac{dt}{t}\\\\
&\leq C\Bint_{Ar-\sqrt{s}}^\infty& \left(\frac{r}{t}\right)^\alpha\frac{1}{|\lambda B_{\sqrt{t^2+s}}|^{1/2}}\left|\left|\sqrt{t^2+s}\nabla e^{-(t^2+s)\Delta}H_{i,t}\right|\right|_{L^2\left(\lambda B_{\sqrt{t^2+s}}\right)}\frac{dt}{t}
\end{array}$$
But there is a constant $c$ such that for every $t\geq r$ and every $s\in[0,nr^2]$, $B_{\sqrt{t^2+s}}\subset B_{ct}$. Choosing $\lambda\leq c^{-1}$ (and $A$ big enough so that $A\lambda\geq1$), and using \eqref{D}, we get

$$\begin{array}{rcl}
&\Bint_{Ar}^\infty&\left(\frac{r}{t}\right)^\alpha\frac{1}{|\lambda B_t|^{1/2}}\left|\left|t\nabla e^{-t^2\Delta}H_{i,\sqrt{t^2-s}}\right|\right|_{L^2(\lambda B_t)}\frac{dt}{t}\\\\
&\leq& C\Bint_{r}^\infty \left(\frac{r}{t}\right)^\alpha\frac{1}{| B_t|^{1/2}}\left|\left|\sqrt{t^2+s}\nabla e^{-(t^2+s)\Delta}H_{i,t}\right|\right|_{L^2\left( B_t\right)}\frac{dt}{t}
\end{array}$$
Using the $L^2$ off-diagonal estimates (Lemma \ref{L2-off}) and Cauchy-Schwarz, we obtain

$$\begin{array}{rcl}

&\Bint_{Ar}^\infty&\left(\frac{r}{t}\right)^\alpha\frac{1}{|\lambda B_t|^{1/2}}\left|\left|t\nabla e^{-t^2\Delta}H_{i,\sqrt{t^2-s}}\right|\right|_{L^2(\lambda B_t)}\frac{dt}{t}\\\\
&\leq& \Bint_r^\infty \left(\frac{r}{t}\right)^\alpha\frac{1}{| B_t|^{1/2}} ||H_{t}||_{L^2\left(C_t^i\right)}e^{-c4^i\left(\frac{r}{t}\right)^2}\frac{dt}{t}\\\\
&\leq& \left(\Bint_r^\infty \left(\frac{r}{t}\right)^\alpha\frac{1}{| B_t|} ||H_{t}||^2_{L^2\left(C_t^i\right)}\frac{dt}{t}\right)^{1/2}\left(\int_r^\infty \left(\frac{r}{t}\right)^\alpha e^{-c4^i\left(\frac{r}{t}\right)^2}\frac{dt}{t}\right)^{1/2}\\\\
&\leq& C4^{-\alpha i} \left(\Bint_r^\infty \left(\frac{r}{t}\right)^\alpha\frac{1}{| B_t|} ||H_{t}||^2_{L^2\left(C_t^i\right)}\frac{dt}{t}\right)^{1/2}
\end{array}$$
We claim that for $\alpha\geq\kappa$,

\begin{equation}\label{max_est_i}
\left(\int_r^\infty \left(\frac{r}{t}\right)^\alpha\frac{1}{| B_t|} ||H_t||^2_{L^2\left(C_t^i\right)}\frac{dt}{t}\right)^{1/2}\leq C 2^{i\kappa/2} \inf_{x\in B}\left(\mathcal{M}\left(|\mathcal{A}_4H|^2\right)(x)\right)^{1/2}.
\end{equation}
Indeed, we compute that

$$\int_{2^{i+2}B}|\mathcal{A}_4H|^2\geq \int_r^\infty \frac{dt}{t}\int_{2^{i+2}B} dx\int_{B(x,4t)}|H(t,y)|^2\frac{dy}{V(y,4t)}.$$
For every $(t,y)\in C_t^i$, there exists $y'\in B(y,4t)$ such that $B\left(y',\frac{r}{2}\right)\subset 2^{i+2}B$ and for every $z\in B\left(y',\frac{r}{2}\right)$, $y\in B(z,4t)$: just take $y'$ lying on a minimizing geodesic going from $x_B$ to $y$, such that $d(y,y')=4t-\frac{r}{2}$. Therefore,

$$\int_{2^{i+2}B}|\mathcal{A}_4H|^2\geq\int_r^\infty \frac{dt}{t}\int_{C^i_t} |H(t,y)|^2\frac{V\left(y',\frac{r}{2}\right)}{V(y,4t)}\,dy.$$
But by Lemma \ref{equiv_vol}, $V(y',4t)\simeq V(y,4t)$, and therefore using \eqref{D}, we get

$$\int_{2^{i+2}B}|\mathcal{A}_4H|^2\geq C\int_r^\infty  \left(\frac{r}{t}\right)^\kappa||H_t||^2_{L^2(C_t^i)}\frac{dt}{t}.$$
By \eqref{D}, if $t\geq r$,

$$\begin{array}{rcl}
\frac{1}{|2^{i+2}B|}&=&\frac{|B_t|}{|2^{i+2}B|}\frac{1}{|B_t|}\\\\
&\geq&\min\left(1,\left(\frac{t}{2^{i+2}r}\right)^\kappa\right)\frac{1}{|B_t|}\\\\
&\geq& C\frac{2^{-i\kappa}}{|B_t|}
\end{array}$$
Therefore, if $\alpha\geq \kappa$,

$$\begin{array}{rcl}
\left(\Bint_r^\infty \left(\frac{r}{t}\right)^\alpha\frac{1}{| B_t|} ||H_t^i||^2_{L^2\left(C_t^i\right)}\frac{dt}{t}\right)^{1/2}&\leq& 2^{i\kappa/2} \left(\frac{1}{|2^{i+2}B|}\Bint_{2^{i+2}B}|\mathcal{A}_4H|^2\right)^{1/2}\\\\
&\leq & C 2^{i\kappa/2} \inf_{x\in B}\left(\mathcal{M}\left(|\mathcal{A}_4H|^2\right)(x)\right)^{1/2},
\end{array}$$
which ends the proof of inequality \eqref{max_est_i}. Now, we can finally prove Lemma \ref{est_max_j}. Indeed, by \eqref{est_max_0} and \eqref{max_est_i}, we have (for our choosen values of $\lambda$ and $A$)

$$\begin{array}{rcl}
&\sum_{i=1}^\infty&  \int_{Ar}^\infty\left(\frac{r}{t}\right)^\alpha\frac{1}{|\lambda B_t|^{1/2}}\left|\left|t\nabla e^{-t^2\Delta}H^i_{\sqrt{t^2-s}}\right|\right|_{L^2(\lambda B_t)}\frac{dt}{t}\\\\
&\leq& C\inf_{x\in B}\left(\mathcal{M}\left(|\mathcal{A}_{16}H|^2\right)(x)\right)^{1/2}\\\\
&&+\left(\sum_{i=2}^\infty 2^{i\kappa/2}4^{-\alpha i} \right)\inf_{x\in B}\left(\mathcal{M}\left(|\mathcal{A}_4H|^2\right)(x)\right)^{1/2}
\end{array}$$
Thus, since $\alpha>\frac{\kappa}{4}$, we get

$$\begin{array}{rcl}
&\sum_{i=1}^\infty&  \int_{Ar}^\infty\left(\frac{r}{t}\right)^\alpha\frac{1}{|\lambda B_t|^{1/2}}\left|\left|t\nabla e^{-t^2\Delta}H^i_{\sqrt{t^2-s}}\right|\right|_{L^2(\lambda B_t)}\frac{dt}{t}\\\\
&\leq& C\inf_{x\in B}\left(\mathcal{M}\left(|\mathcal{A}_{16}H|^2\right)(x)\right)^{1/2}+C\inf_{x\in B}\left(\mathcal{M}\left(|\mathcal{A}_4H|^2\right)(x)\right)^{1/2}
\end{array}$$
By an easy consequence of \eqref{D}, there is a constant $C$ such that for every $H$,

$$\mathcal{A}_4H\leq C\mathcal{A}_{16}H.$$
Therefore,

$$\sum_{i=1}^\infty  \int_{Ar}^\infty\left(\frac{r}{t}\right)^\alpha\frac{1}{|\lambda B_t|^{1/2}}\left|\left|t\nabla e^{-t^2\Delta}H^i_{\sqrt{t^2-s}}\right|\right|_{L^2(\lambda B_t)}\frac{dt}{t}\leq C\inf_{x\in B}\left(\mathcal{M}\left(|\mathcal{A}_{16}H|^2\right)(x)\right)^{1/2}.$$
This shows Lemma \ref{est_max_j}, with $\gamma=16$.

\cqfd

\section{Appendix}
In this Appendix, we prove Proposition \ref{bound_result}. The proof is an adaptation of the proof of [\cite{AC}, Theorem 2.3]. It relies on the following:

\begin{Pro}\label{good}

Let $F$ and $G$ be given, such that for every ball $B$, one can find non-negative, measurable functions $G_B$ and $H_B$ such that

$$F= G_B+H_B\qquad \mbox{ a.e. on }T(B),$$
and satisfying

$$\left(\frac{1}{|B|}\int_B\left\{\mathcal{A}(\chi_{T(B)}H_B)\right\}^p\right)^{1/p}\leq C\left(\inf_{x\in B}\mathcal{C}F(x)+\inf_{x\in B}G(x)\right),$$
and

$$\left(\frac{1}{|B|}\int_{B}\left\{\mathcal{A}\left(\chi_{T(B)}G_B\right)\right\}^2\right)^{1/2}\leq C\inf_{x\in B}G(x).$$
For $2<q<p$, there is a constant $C$ independant of the choice of $F$ and $G$ such that if $||G||_q<\infty$ and $||F||_{T^{2,2}}<\infty$, then

$$||F||_{T^{2,q}}\leq C||G||_q.$$

\end{Pro}
Assuming for the moment the result of Proposition \ref{good}, let us conclude the proof of Proposition \ref{bound_result}. \\\\
\textit{Proof of Proposition \ref{bound_result}:}\\\\
Define $F=T f$, $G_B=T(\chi_{\alpha B}f)$ and $H_B=T\left((1-\chi_{\alpha B})f\right)$. Define finally $G=\left(\mathcal{M}(f^2)\right)^{1/2}$. By linearity of $T$, one has 

$$F= G_B+H_B.$$
By hypothesis,

$$\left(\frac{1}{|B|}\int_B\left\{\mathcal{A}(\chi_{T(B)}H_B)\right\}^p\right)^{1/p}\leq C\left(\inf_{x\in B}\mathcal{C}F(x)+\inf_{x\in B}G(x)\right).$$
We claim that there also holds that

\begin{equation}\label{G_B}
\left(\frac{1}{|B|}\int_{B}\left\{\mathcal{A}\left(\chi_{T(B)}G_B\right)\right\}^2\right)^{1/2}\leq C\inf_{x\in B}G(x).
\end{equation}
Once \eqref{G_B} is shown, the proof of Proposition \ref{bound_result} is completed by applying Proposition \ref{good} to our choice of $F$, $G_B$, $H_B$ and $G$. Indeed, Proposition \ref{good} then gives

$$||F||_{T^{2,q}}\leq C||G||_q.$$
Using the strong $(\frac{q}{2},\frac{q}{2})$ type of the Hardy-Littlewood maximal function, one has

$$||G||_q=\left|\left|\left(\mathcal{M}(f^2)\right)^{1/2}\right|\right|_q\leq C||f||_q,$$
and thus

$$||F||_{T^{2,q}}\leq C||f||_q.$$
By definition of $F$ we obtain

$$||Tf||_{T^{2,q}}\leq C||f||_q.$$
This concludes the proof of Proposition \ref{bound_result}, upon proving inequality \eqref{G_B}. We now proves \eqref{G_B}. Applying successively \eqref{tent_sup2} and \eqref{T22_equiv}, one has

$$\begin{array}{rcl}
\left(\frac{1}{|B|}\int_B\mathcal{A}\left(\chi_{T(B)}G_B\right)^2\right)^{1/2}&\leq& C\left(\frac{1}{|B|}\int_{T(B)} |G_B(t,y)|^2\,dy\frac{dt}{t}\right)^{1/2}\\\\
&\leq& \frac{C}{|B|^{1/2}}||G_B||_{T^{2,2}}.
\end{array}$$
By definition of $G_B$, the fact that $T$ is bounded from $L^2$ to $T^{2,2}$, and \eqref{D}, one has

$$\begin{array}{rcl}
\left(\frac{1}{|B|}\int_B\mathcal{A}\left(\chi_{T(B)}G_B\right)^2\right)^{1/2}&\leq& \frac{C}{|B|^{1/2}} ||\chi_{\alpha B} f||_2\\\\
&\leq&C\left(\frac{1}{|\alpha B|}\int_{\alpha B} |f|^2\right)^{1/2}.
\end{array}$$
Recalling that $G=\left(\mathcal{M}\left(f^2\right)\right)^{1/2}$, we obtain

$$\left(\frac{1}{|B|}\int_B\mathcal{A}\left(\chi_{T(B)}G_B\right)^2\right)^{1/2}\leq \inf_{x\in B} G(x).$$
This completes the proof of \eqref{G_B}, and of Proposition \ref{bound_result}.

\cqfd
It only remains to prove Proposition \ref{good}. We split the proof in a sequence of lemmas. First, we begin by a localization lemma:

\begin{Lem}[Localization]\label{loc}
There exists $K_0$ depending only on the doubling constant, with the following property: if $\bar{x}\in B$ is such that $\mathcal{C}F(\bar{x})\leq \lambda$, then for every $K\geq K_0$,

$$\left\{\chi_B\,\mathcal{C}F>K\lambda\right\}\subset \left\{\mathcal{C}\left(\chi_{T(3B)}F\right)>\frac{K}{K_0}\lambda\right\}.$$

\end{Lem}
\begin{proof}
We first prove that

\begin{equation}\label{C_c}
\mathcal{C}F\leq K_0 \mathcal{C}_cF,
\end{equation}
where $\mathcal{C}_c$ is the centered maximal function defined by

$$\mathcal{C}_cF(x)=\sup_{t>0}\frac{1}{V(x,t)}\int_{T(B(x,t))}|F(t,y)|^2\,dy\frac{dt}{t}.$$
For this, we notice that if $x\in B(x_0,t)$, then $B(x_0,t)\subset B(x,3t)$. Thus, using \eqref{D},

$$\begin{array}{rcl}
\frac{1}{|B(x_0,t)|}\int_{T(B(x_0,t))} |F(t,y)|^2 \,dy\frac{dt}{t}&\leq& \frac{1}{|B(x_0,t)|}\int_{T(B(x,3t))} |F(t,y)|^2 \,dy\frac{dt}{t}\\\\
&\leq& K_0 \frac{1}{|B(x,3t)|}\int_{T(B(x,3t))} |F(t,y)|^2 \,dy\frac{dt}{t}.\\\\
&\leq& K_0 \mathcal{C}_cF(x).
\end{array}$$
Taking the supremum with respect to the set of balls $B(x_0,t)$ containing $x$, we obtain \eqref{C_c}. Now, let $x\in B$ such that $\mathcal{C}F(x)>K\lambda$. By \eqref{C_c}, there is $t>0$ such that

\begin{equation}\label{CF_lambda}
\frac{1}{V(x,t)}\int_{T(B(x,t))}|F(t,y)|^2\,dy\frac{dt}{t}>\frac{K}{K_0}\lambda.
\end{equation}
Furthermore, $\bar{x}\notin B(x,t)$ since $\mathcal{C}F(\bar{x})\leq \lambda\leq \frac{K}{K_0}\lambda$. Therefore, $B(x,t)\subset 3B$. Consequently,

$$\frac{1}{V(x,t)}\int_{T(B(x,t))} |F(t,y)|^2\,dy\frac{dt}{t}=\frac{1}{V(x,t)}\int_{T(B(x,t))} |\chi_{T(3B)}F(t,y)|^2\,dy\frac{dt}{t},$$
which, together with \eqref{CF_lambda}, yields that $\mathcal{C}(\chi_{T(3B)}F)(x)>\frac{K}{K_0}\lambda$.

\end{proof}
We continue with the following good-$\lambda$ inequalities, analogue of Proposition 3.2 in \cite{AC} :

\begin{Lem}[Good-$\lambda$ inequalities]\label{good-lam}

Assume that for every ball $B$, one can find $G_B$ and $H_B$ such that

$$F= G_B+H_B\qquad \mbox{ a.e. on }T(B),$$
and satisfying

$$\left(\frac{1}{|B|}\int_B\left\{\mathcal{A}(\chi_{T(B)}H_B)\right\}^p\right)^{1/p}\leq C\left(\inf_{x\in B}\mathcal{C}F(x)+\inf_{x\in B}G(x)\right),$$
and

$$\left(\frac{1}{|B|}\int_{B}\left\{\mathcal{A}\left(\chi_{T(B)}G_B\right)\right\}^2\right)^{1/2}\leq C\inf_{x\in B}G(x).$$
Then for every $\lambda>0$, $K>K_0$ and $\gamma\leq 1$,

$$\left|\left\{\mathcal{C}F>K\lambda,\,G\leq\gamma \lambda\right\}\right|\leq C\left(\frac{1}{K^p}+\frac{\gamma^2}{K^2}\right)\left|\left\{\mathcal{C}F>\lambda\right\}\right|,$$
provided that $\left\{\mathcal{C}F>\lambda\right\}$ is a proper subset of $M$.

\end{Lem}

\begin{proof}
Let $E_\lambda=\{\mathcal{C}F>\lambda\}$. It is a proper, open subset of $M$. We write a Whitney covering for $E_\lambda$:

$$E_\lambda=\bigcup_i B_i,$$
that is the $B_i$ are balls with the finite intersection property, and there is a constant $c>1$ with the following property: for every $i$, there is $\bar{x}_i\in cB_i$ such that $\bar{x}_i\notin E_\lambda$. Define

$$B_\lambda=\{\mathcal{C}F>K\lambda,\,G\leq \gamma \lambda\}.$$
Then

$$|B_\lambda|\leq \sum_i|B_\lambda\cap B_i|\leq \sum_i|B_\lambda\cap cB_i|.$$
Assume that $B_\lambda\cap cB_i\neq \emptyset$, and let $\bar{y}_i\in B_\lambda\cap cB_i$. By the localisation lemma (Lemma \ref{loc}), 

$$|B_\lambda\cap cB_i|\leq \left|\left\{\mathcal{C}(\chi_{T(3cB_i)}F)>\frac{K}{K_0}\lambda\right\}\right|.$$
Denote $G_i:=G_{3cB_i}$ and $H_i:=H_{3cB_i}$. Then, since $F=G_i+H_i$ on $T(3cB_i)$, there holds

$$\begin{array}{rcl}
\left|\left\{\mathcal{C}(\chi_{T(3cB_i)}F)>\frac{K}{K_0}\lambda\right\}\right|&\leq& \left|\left\{\mathcal{C}(\chi_{T(3cB_i)}G_i)>\frac{K}{2K_0}\lambda\right\}\right|\\\\
&&+\left|\left\{\mathcal{C}(\chi_{T(3cB_i)}H_i)>\frac{K}{2K_0}\lambda\right\}\right|.
\end{array}$$
Let us first estimate the term $\left|\left\{\mathcal{C}(\chi_{T(3cB_i)}G_i)>\frac{K}{2K_0}\lambda\right\}\right|$: using inequality \eqref{tent_max} and the weak $(1,1)$ type of the Hardy-Littlewood maximal function, there holds:

$$\begin{array}{rcl}
\left|\left\{\mathcal{C}(\chi_{T(3cB_i)}G_i)>\frac{K}{2K_0}\lambda\right\}\right|&\leq& \left|\left\{\left(\mathcal{M}\left(\mathcal{A}(\chi_{T(3cB_i)}G_i)\right)^2\right)^{1/2}>C\frac{K}{2K_0}\lambda\right\}\right|\\\\
&\leq& \frac{C}{(K\lambda)^2}||\mathcal{A}(\chi_{T(3cB_i)}G_i)||_2^2 
\end{array}$$
Using the hypothesis on $G_i$, the doubling property \eqref{D}  and the fact that $G(\bar{y}_i)\leq \gamma \lambda$, we get

$$\begin{array}{rcl}
\left|\left\{\mathcal{C}(\chi_{T(3cB_i)}G_i)>\frac{K}{2K_0}\lambda\right\}\right|&\leq& \frac{C}{(K\lambda^2)}|3cB_i|\inf_{x\in 3cB_i} G^2(x)\\\\
&\leq& C\frac{\gamma^2}{K^2}|B_i|.
\end{array}$$
Now, let us estimate $\left|\left\{\mathcal{C}(\chi_{T(3cB_i)}H_i)>\frac{K}{2K_0}\lambda\right\}\right|$: using successively inequality \eqref{tent_max}, the weak $(\frac{p}{2},\frac{p}{2})$ type of the Hardy-Littlewood maximal function $\mathcal{M}$, and the hypothesis on $H_i$, we obtain:

$$\begin{array}{rcl}
\left|\left\{\mathcal{C}(\chi_{T(3cB_i)}H_i)>\frac{K}{2K_0}\lambda\right\}\right|&\leq& \left|\left\{\left(\mathcal{M}\left(\mathcal{A}(\chi_{T(3cB_i)}H_i)\right)^2\right)^{1/2}> C\frac{K}{2K_0}\lambda\right\}\right|\\\\
&\leq& \frac{C}{(K\lambda)^p}\Bint_{T(3cB_i)}\left|\mathcal{A}(\chi_{T(3cB_i)}H_i)\right|^p\\\\
&\leq&\frac{C}{(K\lambda)^p}|B_i|\left(\frac{1}{|3cB_i|}\Bint_{T(3cB_i)}\left|\mathcal{A}(\chi_{T(3cB_i)}H_i)\right|^p\right)\\\\
&\leq&\frac{C}{(K\lambda)^p}|B_i|\left(\inf_{x\in 3cB_i}\mathcal{C}F(x)+\inf_{x\in 3cB_i}G(x)\right)^{p/2}\\\\
&\leq& \frac{C}{(K\lambda)^p}|B_i|\left(\mathcal{C}F(\bar{x}_i)+G(\bar{y}_i)\right)^p\\\\
&\leq& \frac{C}{K^p}|B_i|.
\end{array}$$
Finally, we get, using the finite intersection property of the balls $B_i$,

$$|B_\lambda|\leq C\left(\frac{\gamma^2}{K^2}+\frac{1}{K^p}\right)\sum_i|B_i|\leq C\left(\frac{\gamma^2}{K^2}+\frac{1}{K^p}\right) |E_\lambda|,$$
and the proof of Lemma \ref{good-lam} is complete.

\end{proof}
\noindent\textit{End of the proof of Proposition \ref{good}:}\\\\
Define

$$\Phi(t)=q\int_0^t \lambda^{q-1}|\{\mathcal{C}F>\lambda\}|\,d\lambda.$$
By \eqref{tent_max}, 

$$|\{\mathcal{C}F>\lambda\}|\leq |\{\left(\mathcal{M}(\mathcal{A}F^2\right)^{1/2}>C\lambda\}|,$$
and by the weak $(1,1)$ type of $\mathcal{M}$,

$$\begin{array}{rcl}
|\{\mathcal{C}F>\lambda\}|&\leq& \frac{C}{\lambda^2}||\mathcal{A}F||_2^2\\\\
&\leq& \frac{C}{\lambda^2}||F||_{T^{2,2}}^2\\\\
&<&\infty
\end{array}$$
Since \eqref{D} implies that $M$ has infinite volume, necessarily $\{\mathcal{C}F>\lambda\}$ is a proper subset of $M$. Therefore, we can apply Lemma \ref{good-lam}, and obtain by integration that

$$\Phi(Kt)\leq C K^q \left(\frac{1}{K^p}+\frac{\gamma^2}{K^2}\right)\Phi(t)+\left(\frac{K}{\gamma}\right)^q||G||_q^q.$$
Since $q<p$, one can choose $K$ large enough and $\gamma$ small enough so that

$$K^q \left(\frac{1}{K^p}+\frac{\gamma^2}{K^2}\right)\leq \frac{1}{2}.$$
Hence, for this choice, we get that for all $t\geq0$,

$$\Phi(Kt)\leq \frac{1}{2}\Phi(t)+\left(\frac{K}{\gamma}\right)^q||G||_q^q.$$
By an easy iteration, this implies that for every $t\geq0$,

$$\Phi(t)\leq C ||G||_q^q.$$
Since $\lim_{t\to\infty}\Phi(t)=||\mathcal{C}F||_q^q$, we get that

$$||\mathcal{C}F||_q\leq C ||G||_q.$$
But by Theorem \ref{CMS}, since $2<q$ there holds that

$$||\mathcal{C}F||_q\simeq ||F||_{T^{2,q}},$$
hence the result of Proposition \ref{good}.

\cqfd

\begin{Rem}

{\em
On $T^{2,q}$ for $q>2$, there are two ``maximal" functions, namely $\mathcal{C}$ and $\mathcal{G}:F\mapsto \left(\mathcal{M}(|\mathcal{A}F|^2)\right)^{1/2}$, with the property that

$$||F||_{T^{2,q}}\simeq ||\mathcal{C}F||_q\simeq ||\mathcal{G}F||_{q/2}.$$
The proof of Proposition \ref{good} with $\mathcal{G}$ instead of $\mathcal{C}$ works fine, except at one (crucial) place: the localisation lemma (Lemma \ref{loc}), which is not true for $\mathcal{G}$. This explains our choice of $\mathcal{C}$ in the proof of Proposition \ref{good}.
}

\end{Rem}


\begin{thebibliography}{99}
%
\bibitem{AMR} P.~Auscher, A.~McIntosh and E.~Russ, Hardy spaces of differential forms on Riemannian manifolds, {\em J. Geom. Anal.} {\bf 18} (2008), no. 1, 192–-248.

\bibitem{ACDH} P.~Auscher, T.~Coulhon, X.T.~Duong, S.~Hofmann, Riesz transform on manifolds and heat kernel regularity, {\em Ann. Sci. \'{E}cole Norm. Sup.} (4) {\bf 37} (2004), no. 6, 911–-957.

\bibitem{AC} P.~Auscher, T.~Coulhon, Riesz transform on manifolds and Poincar\'{e} inequalities, {\em Ann. Sc. Norm. Super. Pisa Cl. Sci.} (5) {\bf 4} (2005), no. 3, 531–-555.

\bibitem{CD} T.~Coulhon, X.T.~Duong, Riesz transforms for $1\leq p\leq 2$, {\em Trans. Amer. Math. Soc.} {\bf 351} (1999), no. 3, 1151-–1169.

\bibitem{CZ} T.~Coulhon, Q.S.~Zhang, Large time behavior of heat kernels on forms,  
{\em J. Differential Geom.} {\bf 77} (2007), no. 3, 353–-384.

\bibitem{CMS} R.R.~Coifman, Y.~Meyer, E.M.~Stein, Some new function spaces and their applications to harmonic analysis, {\em J. Funct. Anal.} {\bf 62} (1985), no. 2, 304–335.

\bibitem{D} B.~Devyver, A gaussian estimate for the heat kernel on differential forms and application to the Riesz transform, {\em to appear in Math. Annalen}

\bibitem{F} W.~Feller, An Introduction to Probability Theory and its Applications, vol. I, Wiley, New York, 1968.

\bibitem{HLMMY} S.~Hofmann, G.~Lu, D.~Mitrea, M.~Mitrea, L.~Yan, Hardy spaces associated to non-negative self-adjoint operators satisfying Davies-Gaffney estimates, {\em Mem. Amer. Math. Soc.} {\bf 214} (2011), no. 1007.

\bibitem{S} Z.~Shen, Bounds of Riesz transforms on $L^p$ spaces for second order elliptic operators, {\em Ann. Inst. Fourier (Grenoble)} {\bf 55} (2005), no. 1, 173–-197.

\bibitem{St} E.M.~Stein, Harmonic analysis: real-variable methods, orthogonality, and oscillatory integrals,  {\em Princeton Mathematical Series}, {\bf 43}, Princeton University Press, Princeton, NJ, 1993.

\bibitem{St2} E.M.~Stein, Singular Integrals and Differentiability of Functions, Princeton University Press, Princeton (1970).

\bibitem{Y} S.T.~Yau, Some function-theoretic properties of complete Riemannian manifold and their applications to geometry, {\em Indiana Univ. Math. J.} {\bf 25} (1976), no. 7, 659-–670.

\end{thebibliography}
\end{document}